\documentclass[11pt,a4paper]{amsart}
\usepackage[utf8]{inputenc}
\usepackage[T1]{fontenc}
\usepackage{ae,aecompl}
\usepackage[english]{babel}
\usepackage{enumerate}
\usepackage{latexsym}
\usepackage{amsmath,amsthm,amsfonts,amssymb}
\usepackage{xr}
\usepackage{graphicx}
\usepackage{stmaryrd}
\usepackage{dsfont}
\usepackage{mathtools}
\usepackage{braket}
\usepackage{color}
\usepackage{hyperref}
\usepackage{relsize}
\usepackage{eurosym, eucal,palatino}
\usepackage{scrextend}
\usepackage[normalem]{ulem}

\changefontsizes{10.1pt}

\newtheorem{thm}{Theorem}[section]

\newtheorem{cor}[thm]{Corollary}
\newtheorem{lem}[thm]{Lemma}
\newtheorem{prop}[thm]{Proposition}

\newtheorem{bem}[thm]{Remark}

\newcommand{\sgn}{\operatorname{sgn}}

\newcommand{\dom}{\operatorname{dom}}

\newcommand{\supp}{\operatorname{supp}}

\numberwithin{equation}{section}

\renewcommand\and{\qquad\text{and}\qquad}

\newcommand\sm{\setminus}
\newcommand\fra{\mathfrak{a}}

\newcommand\dl{\delta}

\newcommand{\comm}[1]{}

\renewcommand\aa{\alpha}
\newcommand\lm{\lambda}

\newcommand\s{\sigma}

\newcommand\omg{\omega}

\newcommand\arr{\rightarrow}

\newcommand\sess{\sigma_{\rm ess}}
\newcommand\dd{{\mathsf{d}}}

\newcounter{counter_a}
\newenvironment{myenum}{\begin{list}{{\rm(\roman{counter_a})}}%
{\usecounter{counter_a}
\setlength{\itemsep}{1.ex}\setlength{\topsep}{0.8ex}
\setlength{\leftmargin}{5ex}\setlength{\labelwidth}{5ex}}}{\end{list}}

\newcommand\wt{\widetilde}

      \def\dC{{\mathbb C}}

      \def\dR{{\mathbb R}}

\def\sfA{{\mathsf A}}   \def\sfB{{\mathsf B}}   
\def\sfD{{\mathsf D}}      
      \def\sfI{{\mathsf I}}

   \def\sfT{{\mathsf T}}   
   \def\sfW{{\mathsf W}}

\def\cA{{\mathcal A}}   \def\cB{{\mathcal B}}   
      \def\cF{{\mathcal F}}

\makeatletter
\def\section{\@startsection{section}{1}\z@{.9\linespacing\@plus\linespacing}%
	{.7\linespacing} {\fontsize{13}{14}\selectfont\bfseries\centering}}
\def\paragraph{\@startsection{paragraph}{4}%
	\z@{0.3em}{-.5em}%
	{$\bullet$ \ \normalfont\itshape}}

\@addtoreset{equation}{section}
\makeatother

\title[]{Schr\"odinger operators with \boldmath{$\dl$}-potentials supported on unbounded Lipschitz hypersurfaces}

\author[J. Behrndt]{Jussi Behrndt}
\address{(JB) Institute of Applied Mathematics, Graz University of Technology, 8010 Graz, Austria}
\email{behrndt@tugraz.at}

\author[V. Lotoreichik]{Vladimir Lotoreichik}
\address{(VL) Department of Theoretical Physics, Nuclear Physics Institute, Czech Academy of Sciences, 25068 \v{R}e\v{z}, Czech Republic}
\email{lotoreichik@ujf.cas.cz }

\author[P. Schlosser]{Peter Schlosser}
\address{(PS) Institute of Applied Mathematics, Graz University of Technology, 8010 Graz, Austria}
\email{schlosser@tugraz.at}

\begin{document}

\maketitle

\dedicatory{\hspace{6ex}\it Dedicated to the memory of our friend and colleague Sergey Naboko}

\begin{abstract}
In this note we consider the self-adjoint Schr\"odinger operator $\sfA_\aa$ in $L^2(\dR^d)$, $d\geq 2$, with a $\delta$-potential supported on a Lipschitz hypersurface $\Sigma\subseteq\mathbb{R}^d$ of strength $\aa\in L^p(\Sigma)+L^\infty(\Sigma)$.
We show the uniqueness of the ground state and, under some additional
conditions on the coefficient $\alpha$ and the hypersurface $\Sigma$, we determine the essential spectrum of $\sfA_\aa$. In the special case that $\Sigma$ is a hyperplane we 
obtain a Birman-Schwinger principle with a relativistic Schr\"{o}dinger operator as 
Birman-Schwinger operator. As an application we prove an optimization result for the bottom of
the spectrum of $\sfA_\aa$.
\end{abstract}

\section{Introduction}

In this paper we are interested in spectral properties of a 
class of self-adjoint operators $\sfA_\aa$ with singular $\delta$-potentials 
in the Hilbert 
space $L^2(\dR^d)$, $d\ge 2$, which correspond to the formal differential expression
\begin{equation}\label{formal}
-\Delta-\aa\,\delta(x-\Sigma),
\end{equation}
where $\Sigma\subset\dR^d$ is the graph of a Lipschitz function $\xi:\dR^{d-1}\rightarrow\dR$ and $\aa\colon\Sigma\arr\dR$ is the strength of the $\delta$-potential; cf. \cite{BLL13, BEKS94}, the monograph~\cite{EK} and the references therein. Note that 
the unbounded Lipschitz surface $\Sigma$ splits $\dR^d$ into two unbounded disjoint parts and that the special choice $\xi=0$ corresponds to the situation where $\Sigma$ is the hyperplane in $\dR^d$. Assuming $\alpha\in L^p(\Sigma)+L^\infty(\Sigma)$ for some $1<p<\infty$ in $d=2$ and for $d-1\leq p<\infty$ in $d\geq 3$ dimensions 
we define $\sfA_\aa$ as the semibounded self-adjoint operator in $L^2(\dR^d)$
associated with the densely defined, symmetric, semibounded, 
and closed form 
\begin{equation}\label{Eq_Schroedinger_form}
\begin{split}
\mathfrak{a}_\alpha[u,v]&\coloneqq(\nabla u,\nabla v)_{L^2(\mathbb{R}^d;{\mathbb{C}^d})}-\int_{\Sigma}\alpha\,\tau_{\rm D}u\,\overline{\tau_{\rm D}v}\,\dd x, \\
\dom\mathfrak{a}_\alpha&\coloneqq H^1(\mathbb{R}^d),
\end{split}
\end{equation}
where $\tau_{\rm D}:H^1(\mathbb{R}^d)\rightarrow H^{1/2}(\Sigma)$ is the Dirichlet trace operator. Let us denote the bottom of the spectrum of $\sfA_\alpha$ by
\begin{equation}\label{eq:lowest}
\lambda_1(\alpha)\coloneqq\inf\sigma(\sfA_\alpha).
\end{equation}

The first issue we discuss in this paper is the essential spectrum of the self-adjoint operator $\sfA_\aa$. 
In the present situation one always has the inclusion $[0,\infty)\subset \sess(\sfA_\alpha)$ and in Theorem \ref{thm:Essential_spectrum}
we prove that if $\Sigma$ is a local deformation of the hyperplane $\dR^{d-1}\times \{0\}$ and $\alpha$ is close to 
a constant $\alpha_0$ outside of sets of finite measure, then
\begin{equation*}
\sess(\sfA_\alpha)=\left\{\begin{array}{ll} [-\frac{\alpha_0^2}{4},\infty), & \text{if }\alpha_0\geq 0, \\ {[0,\infty)}, & \text{if }\alpha_0\leq 0; \end{array}\right.
\end{equation*}
see also \cite{R19} for related results. Next we investigate 
the uniqueness of the ground state of $\sfA_\aa$, which is a typical property for Schr\"{o}dinger operators $-\Delta+V$ with {\it regular} potentials.
More precisely, if $\lambda_1(\alpha)$ in \eqref{eq:lowest} is a discrete eigenvalue then 
it will be shown in Section~\ref{sec_AppA} that $\lambda_1(\alpha)$ is simple and the corresponding 
eigenfunction can be chosen strictly positive 
on $\mathbb{R}^d\setminus\Sigma$; this observation is based on a standard argument using Harnack's inequality.

In Section~\ref{sec_BS} we focus on the special case that $\Sigma$ is the hyperplane and we 
obtain a Birman-Schwinger principle, where the Birman-Schwinger operator is a relativistic 
Schr\"{o}dinger operator in $L^2(\dR^{d-1})$. The operators appearing in this context can also be viewed as (extensions of) the 
$\gamma$-field and Weyl function corresponding to a certain quasi boundary triple; cf. \cite[Section~8]{BLLR18} for more details. Under the additional assumption that $\alpha$ is close to 
a constant outside of sets of finite measure we 
then provide a more detailed analysis of the spectrum
of the Birman-Schwinger operator and link these spectral properties to those of $\sfA_\aa$.
As an interesting application we prove an optimization result 
for the bottom of the spectrum of $\sfA_\alpha$ which is formulated in terms of the so-called  
\textit{symmetric decreasing rearrangement}: Consider again a real-valued $\aa\in L^\infty(\dR^{d-1})+L^p(\dR^{d-1})$ 
for some $1<p<\infty$ in $d=2$ and for $d-1\leq p<\infty$ in $d\geq 3$ dimensions, and assume that $\alpha$ is close to 
a constant $\alpha_0\in\mathbb{R}$ outside of sets of finite measure. Furthermore, let $\aa_1\coloneqq\aa-\aa_0$ and $(\aa_1)_+=\max\{\aa_1,0\}$, and let $(\aa_1)_+^*$ be 
the symmetric decreasing rearrangement of $(\aa_1)_+$ defined in \eqref{Eq_Symmetric_decreasing_rearrangement}. Then we have the inequality
\begin{equation}\label{opti}
\lm_1(\aa_0+(\aa_1)_+^*)\le \lm_1(\aa_0+\aa_1).
\end{equation}
Our proof of \eqref{opti} relies on the fact that the symmetric decreasing rearrangement decreases the kinetic energy term corresponding to the relativistic Schr\"odinger operator. This property of the kinetic energy can be viewed as an analogue of the P\'{o}lya-Szeg\H{o} inequality. We note that a different argument for 
\eqref{opti} based on Steiner symmetrization was communicated to us; cf. Remark~\ref{otherbem} for more details.
We wish to mention that eigenvalue optimization is a trademark topic in spectral theory; see the monographs~\cite{Henrot,Henrot2} and the references therein. In particular,
optimization of eigenvalues induced by $\dl$-potentials supported on hypersurfaces is a topic of growing interest~\cite{E05, EHL06, EL17, L19}. There are also closely related works on eigenvalue optimization for $\dl$-potentials supported on sets of higher co-dimension~\cite{BFKLR17, EK19}, for the Robin Laplacian~\cite{AFK17, B86, BFNT18, D06, FK15, FL18, GL19, KL19, KL20}, for $\dl'$-interactions~\cite{L18} and for Dirac operators with surface 
interactions~\cite{ABLO20,AMV16}.

 \vskip 0.2cm\noindent\\
    {\bf Acknowledgements}. 
    J. Behrndt gratefully acknowledges financial support by the Austrian Science Fund (FWF): P 33568-N.
    V. Lotoreichik was supported 
    	by the Czech Science Foundation project 21-07129S. 
    	This publication is based upon work from COST Action CA 18232 
MAT-DYN-NET, supported by COST (European Cooperation in Science and 
Technology), www.cost.eu.

\section{The Schrödinger operator with $\delta$-potential supported on a Lipschitz graph}\label{sec_Schroedinger_operator}

In this section let $d\geq 2$ and 
\begin{equation}\label{Eq_Sigma}
\Sigma\coloneqq\Set{(x,\xi(x)) | x\in\mathbb{R}^{d-1}}\subset\mathbb{R}^d
\end{equation}
be the graph of a Lipschitz function $\xi:\mathbb{R}^{d-1}\rightarrow\mathbb{R}$. Furthermore, let
\begin{equation}\label{Eq_alpha}
\alpha\in L^p(\Sigma)+L^\infty(\Sigma)
\end{equation}
be a real-valued function with $1<p<\infty$ in $d=2$ and $d-1\leq p<\infty$ in $d\geq 3$ dimensions. In this setting we will define the self-adjoint operator $\sfA_\alpha$ associated to the form \eqref{Eq_Schroedinger_form} and study its essential spectrum. In particular, under some additional flatness assumptions on the support $\Sigma$ and some decay at infinity of the coefficient $\alpha$ we explicitly compute $\sigma_\text{ess}(\sfA_\alpha)$. Furthermore, we verify that the ground state $\lambda_1(\alpha)$ (if it is a discrete eigenvalue) is simple. 

\subsection{The form $\mathfrak{a}_\alpha$ and the operator $\sfA_\alpha$}

In this subsection we will prove that the form \eqref{Eq_Schroedinger_form}, which models a $\delta$-potential of strength $\alpha$ supported on $\Sigma$, is well defined and gives rise to a self-adjoint operator $\sfA_\alpha$ in $L^2(\mathbb{R}^d)$; cf. \cite{BEKS94,FL08} and \cite[Proposition~3.8]{BS19}. In the following the Dirichlet trace operator $\tau_{\rm D}$ in \eqref{Eq_Schroedinger_form} is viewed for $\frac{1}{2}<s<\frac{3}{2}$ as a bounded operator
\begin{equation}\label{Eq_Trace_Sobolev_boundedness}
\tau_{\rm D}\colon H^s(\mathbb{R}^d)\rightarrow H^{s-\frac{1}{2}}(\Sigma);
\end{equation}
cf. \cite[Proof of Theorem  3.38]{McLean}.

\begin{prop}\label{prop_Closed_form}
The form $\mathfrak{a}_\alpha$ in \eqref{Eq_Schroedinger_form} is densely defined, symmetric, semibounded, 
and closed in $L^2(\dR^d)$.
\end{prop}

\begin{proof}
It is clear, that $\dom\mathfrak{a}_\alpha=H^1(\mathbb{R}^d)$ is dense in $L^2(\mathbb{R}^d)$. Furthermore, we split $\mathfrak{a}_\alpha$ into
\begin{align*}
\mathfrak{a}_0[u,v]&\coloneqq(\nabla u,\nabla v)_{L^2(\mathbb{R}^d;\mathbb{C}^d)},&\text{with}&&\dom\mathfrak{a}_0&\coloneqq H^1(\mathbb{R}^d), \\
\mathfrak{a}_1[u,v]&\coloneqq-\int_\Sigma\alpha\,\tau_{\rm D}u\,\overline{\tau_{\rm D}v}\,\dd x,&\text{with}&&\dom\mathfrak{a}_1&\coloneqq H^1(\mathbb{R}^d).
\end{align*}
Observe that $\mathfrak{a}_0$ is densely defined, nonnegative, and closed in $L^2(\dR^d)$. Furthermore, since $\alpha$ is real-valued it is clear that $\mathfrak{a}_1$ is symmetric. The estimate \eqref{Eq_Malpha_boundedness} shows that for every $\varepsilon>0$ there exists some $c_\varepsilon\geq 0$, such that
\begin{equation*}
\big|\mathfrak{a}_1[u,u]\big|\leq\varepsilon^2\Vert\tau_{\rm D}u\Vert^2_{H^{\frac{1}{2}}(\Sigma)}+c_\varepsilon^2\Vert\tau_{\rm D}u\Vert^2_{L^2(\Sigma)},\qquad u\in H^1(\mathbb{R}^d).
\end{equation*}
Using the boundedness \eqref{Eq_Trace_Sobolev_boundedness} of the trace operator, the absolute value of $\mathfrak{a}_1[u,u]$ can further be estimated by
\begin{equation*}
\big|\mathfrak{a}_1[u,u]\big|\leq\varepsilon^2d_1^2\Vert u\Vert^2_{H^1(\mathbb{R}^d)}+c_\varepsilon^2d_s^2\Vert u\Vert^2_{H^s(\mathbb{R}^d)},\qquad u\in H^1(\mathbb{R}^d),
\end{equation*}
where $d_1$ and $d_s$ are the operator norms \eqref{Eq_Trace_Sobolev_boundedness} with $s=1$ and  some fixed 
$s\in(\frac{1}{2},1)$, respectively. Since $s<1$, we can use \cite[Theorem 3.30]{HT08} to find a constant $\tilde{c}_\varepsilon\geq 0$ with
\begin{equation*}
\big|\mathfrak{a}_1[u,u]\big|\leq\varepsilon^2(d_1^2+1)\Vert u\Vert^2_{H^1(\mathbb{R}^d)}+\tilde{c}_\varepsilon^2\Vert u\Vert^2_{L^2(\mathbb{R}^d)},\qquad u\in H^1(\mathbb{R}^d).
\end{equation*}
That is, the form $\mathfrak{a}_1$ is $\mathfrak{a}_0$-bounded with form bound $0$. The semiboundedness and closedness of $\mathfrak{a}_\alpha=\mathfrak{a}_0+\mathfrak{a}_1$ now follow from \cite[Chapter {VI}, Theorem 1.33]{Ka1976}.
\end{proof}

Proposition \ref{prop_Closed_form} combined with the First Representation Theorem \cite[Chapter~VI, Theorem~2.1]{Ka1976} implies that there is a unique self-adjoint operator $\sfA_\alpha$ in $L^2(\dR^d)$ representing the form $\mathfrak{a}_\alpha$ in
the sense that $\dom \sfA_\alpha\subset \dom\fra_\aa$ and 
\begin{equation}\label{dada}
(\sfA_\alpha f,g)_{L^2(\mathbb{R}^d)}=\mathfrak{a}_\alpha[f,g],\quad f\in\dom \sfA_\alpha,\, g\in \dom\fra_\aa.
\end{equation}

\subsection{Essential spectrum of $\sfA_\alpha$}

In this subsection we investigate the essential spectrum of $\sfA_\alpha$. The following preparatory lemma shows that in the present situation the essential spectrum of $\sfA_\alpha$ always covers the nonnegative real axis.

\begin{lem}\label{lemmilein}
For any $\alpha$ of the form \eqref{Eq_alpha} we have
\begin{equation}\label{Eq_Essential_spectrum_inclusion2}
[0,\infty)\subseteq\sess(\sfA_\alpha).
\end{equation}
\end{lem}

\begin{proof}
In a similar way as in the proof of \cite[Theorem 6.5]{EE18} one constructs for 
$\lambda\in(0,\infty)$ an orthonormal sequence $(\Psi_n)_n\in\mathcal{C}_0^\infty(\mathbb{R}^d)$ with support in $\mathbb{R}^d\setminus\Sigma$ and
\begin{equation*}
\Vert(-\Delta-\lambda)\Psi_n\Vert_{L^2(\mathbb{R}^d)}\overset{n\rightarrow\infty}{\longrightarrow}0.
\end{equation*}
From $\supp\Psi_n\subseteq\mathbb{R}^d\setminus\Sigma$ we have $\tau_{\rm D}\Psi_n=0$ and hence it follows from 
\eqref{Eq_Schroedinger_form} that $\sfA_\alpha\Psi_n=-\Delta\Psi_n$. This implies
\begin{equation*}
\Vert(\sfA_\alpha-\lambda)\Psi_n\Vert_{L^2(\mathbb{R}^d)}\overset{n\rightarrow\infty}{\longrightarrow}0,
\end{equation*}
so that $(\Psi_n)_n$ is a singular sequence and we conclude $\lambda\in\sess(\sfA_\alpha)$. This proves that $(0,\infty)\subseteq\sess(\sfA_\alpha)$ and since the essential spectrum is closed we obtain \eqref{Eq_Essential_spectrum_inclusion2}.
\end{proof}

For a subclass of hypersurfaces $\Sigma$, which are local deformations of a hyperplane, and interaction strengths $\alpha$ having a certain decay at infinity, we are able to determine the essential spectrum explicitly.

\begin{thm}\label{thm:Essential_spectrum}
If the function $\xi:\mathbb{R}^{d-1}\rightarrow\mathbb{R}$ in \eqref{Eq_Sigma} is compactly supported
and if for some $\alpha_0\in\mathbb{R}$
\begin{equation}\label{Eq_Decay_property_potential}
\Set{x\in\Sigma | \vert\alpha(x)-\alpha_0\vert>\varepsilon}\text{ has finite measure for every }\varepsilon>0,
\end{equation}
then the essential spectrum of the corresponding Schrödinger operator $\sfA_\alpha$ is given by
\begin{equation}\label{Eq_Essential_spectrum}
\sess(\sfA_\alpha)=\left\{\begin{array}{ll} [-\frac{\alpha_0^2}{4},\infty), & \text{if }\alpha_0\geq 0, \\ {[0,\infty)}, & \text{if }\alpha_0\leq 0. \end{array}\right.
\end{equation}
\end{thm}

\begin{proof}
{\it Step 1.} First, we consider the hyperplane $\Sigma=\mathbb{R}^{d-1}\times\{0\}\cong\mathbb{R}^{d-1}$ and the constant potential $\alpha(x)=\alpha_0$. We introduce two auxiliary closed forms
\begin{align*}
\mathfrak{d}[\phi,\psi]&\coloneqq(\nabla\phi,\nabla\psi)_{L^2(\mathbb{R}^{d-1};\mathbb{C}^{d-1})},&\text{with}\quad&\dom\mathfrak{d}\coloneqq H^1(\mathbb{R}^{d-1}), \\
\mathfrak{t}_{\alpha_0}[f,g]&\coloneqq(f',g')_{L^2(\mathbb{R})}-\alpha_0f(0)\overline{g(0)},&\text{with}\quad&\dom\mathfrak{t}_{\alpha_0}\coloneqq H^1(\mathbb{R}),
\end{align*}
with the corresponding self-adjoint operators $-\Delta$ and $\sfT_{\aa_0}$ in the Hilbert spaces $L^2(\dR^{d-1})$ and $L^2(\dR)$, respectively. The spectra of these operators are explicitly given by
\begin{equation*}
\s(-\Delta)=[0,\infty)\quad\text{and}\quad\s(\sfT_{\aa_0})=\left\{\begin{array}{ll} \{-\frac{\alpha_0^2}{4}\}\cup[0,\infty), & \text{if }\alpha_0\geq 0, \\ {[0,\infty)}, & \text{if }\alpha_0\leq 0, \end{array}\right.
\end{equation*}
where the proof of the latter one can be found in \cite[Theorem 3.1.4]{AGHH05}. The Schrödinger operator $\wt\sfA_{\aa_0}$ with $\delta$-potential supported on a hyperplane of constant strength $\aa_0$ can be decomposed as
\begin{equation*}
\wt\sfA_{\aa_0}=(-\Delta)\otimes\sfI_\mathbb{R}+\sfI_{\mathbb{R}^{d-1}}\otimes\sfT_{\aa_0}
\end{equation*}
with respect to $L^2(\dR^d)=L^2(\dR^{d-1})\otimes L^2(\dR)$; here $\sfI_\mathbb{R}$ and $\sfI_{\mathbb{R}^{d-1}}$ denote the identity operators in $L^2(\mathbb{R})$ and $L^2(\mathbb{R}^{d-1})$, respectively. Hence, it follows from \cite[Eq. (4.44)]{Te2009} that
\begin{equation}\label{eq:specAaa0}
\s(\wt\sfA_{\aa_0})=\left\{\begin{array}{ll} [-\frac{\alpha_0^2}{4},\infty), & \text{if }\alpha_0\geq 0, \\ {[0,\infty)}, & \text{if }\alpha_0\leq 0. \end{array}\right.
\end{equation}

{\it Step 2.} Let $\sfA_{\aa_0}$ be the Schrödinger operator with $\delta$-potential of constant strength $\aa_0$ supported on the hypersurface $\Sigma$. Since the Lipschitz mapping $\xi$ is compactly supported, the surface $\Sigma$ is a local deformation of the hyperplane $\dR^{d-1}\times\{0\}$ in the sense that $\Sigma\sm\cB=(\dR^{d-1}\times\{0\})\sm\cB$ for a ball $\cB\subset\dR^d$ of sufficiently large radius. Hence it follows from \eqref{eq:specAaa0} using \cite[Theorem 4.7]{BEL14} that 
\begin{equation}\label{eq:essspec}
\sess(\sfA_{\aa_0})=\sess(\wt\sfA_{\aa_0})=\left\{\begin{array}{ll} [-\frac{\alpha_0^2}{4},\infty), & \text{if }\alpha_0\geq 0, \\ {[0,\infty)}, & \text{if }\alpha_0\leq 0. \end{array}\right. 
\end{equation}

{\it Step 3.} With $\alpha_0$ from the decay property \eqref{Eq_Decay_property_potential} we define $\alpha_1\coloneqq\alpha-\alpha_0$, such that $\Set{x\in\Sigma | \vert\alpha_1(x)\vert>\varepsilon}$ has finite measure for every $\varepsilon>0$. The self-adjoint operators $\sfA_{\aa_0}$ and $\sfA_\aa$ are both semibounded since they correspond to semibounded forms. Hence, we can fix $\lambda<\inf(\s(\sfA_{\aa_0})\cup\sigma(\sfA_\aa))$ and consider the resolvent difference
\begin{equation}\label{eq:W}
\sfW\coloneqq(\sfA_{\aa_0}-\lambda)^{-1}-(\sfA_\aa-\lambda)^{-1}.
\end{equation}
Our aim is to show that $\sfW$ is a compact operator in $L^2(\mathbb{R}^d)$. For this let $f,g\in L^2(\mathbb{R}^d)$ and set
\begin{equation}\label{eq:uv}
u\coloneqq(\sfA_{\aa_0}-\lambda)^{-1}f\quad\text{and}\quad v\coloneqq(\sfA_\aa-\lambda)^{-1}g.
\end{equation}
Using \eqref{eq:uv} and the definition of the operator $\sfW$ in~\eqref{eq:W} we obtain
\begin{equation}\label{Eq_W_reduction}
\begin{split}
(\sfW f,g)_{L^2(\dR^d)}&=\big((\sfA_{\aa_0}-\lambda)^{-1}f,g\big)_{L^2(\dR^d)}-\big((\sfA_\aa-\lambda)^{-1}f,g\big)_{L^2(\dR^d)} \\
&=(u,g)_{L^2(\dR^d)}-(f,v)_{L^2(\dR^d)} \\
&=\big(u,(\sfA_\aa-\lambda)v\big)_{L^2(\dR^d)}-\big((\sfA_{\aa_0}-\lambda)u,v\big)_{L^2(\dR^d)} \\
&=(u,\sfA_\aa v)_{L^2(\dR^d)}-(\sfA_{\aa_0}u,v)_{L^2(\dR^d)}.
\end{split}
\end{equation}
We can express the above inner products via the corresponding forms \eqref{dada} and conclude that $(\sfW f,g)_{L^2(\dR^d)}$ reduces to the surface integral
\begin{equation*}
(\sfW f,g)_{L^2(\dR^d)}=-\int_\Sigma\alpha_1\tau_{\rm D}u\,\overline{\tau_{\rm D}v}\,\dd x=(\sfT_1f,\sfT_2g)_{L^2(\Sigma)},
\end{equation*}
where $\sfT_1,\sfT_2:L^2(\mathbb{R}^d)\rightarrow L^2(\Sigma)$ are defined by
\begin{equation*}
\sfT_1\coloneqq|\alpha_1|^{\frac{1}{2}}\tau_{\rm D}(\sfA_{\aa_0}-\lambda)^{-1}\quad\text{and}\quad\sfT_2\coloneqq-\sgn(\alpha_1)|\alpha_1|^{\frac{1}{2}}\tau_{\rm D}(\sfA_\aa-\lambda)^{-1}.
\end{equation*}
As $(\sfA_{\alpha_0}-\lambda)^{-1}$ and $(\sfA_\alpha-\lambda)^{-1}$ are bounded operators from $L^2(\mathbb{R}^d)$ into $H^1(\mathbb{R}^d)$, it follows from \eqref{Eq_Trace_Sobolev_boundedness} that $\tau_{\rm D}(\sfA_{\alpha_0}-\lambda)^{-1}$ and $\tau_{\rm D}(\sfA_\alpha-\lambda)^{-1}$ are bounded from $L^2(\mathbb{R}^d)$ into $H^{\frac{1}{2}}(\Sigma)$. Consequently, both $\sfT_1$ and $\sfT_2$ are compact as operators from $L^2(\mathbb{R}^d)$ into $L^2(\Sigma)$ by Proposition \ref{prop_Malpha_compactness}. Thus the operator $\sfW=\sfT_2^*\sfT_1$ is compact as well and the stability of the essential spectrum under compact perturbations in resolvent sense 
combined with~\eqref{eq:essspec} yields the claim.
\end{proof}

\subsection{Uniqueness of the ground state}\label{sec_AppA}

In this subsection we assume that the bottom of the spectrum $\lambda_1(\alpha)$ in \eqref{eq:lowest} is a discrete eigenvalue of $\sfA_\alpha$. The aim is to prove in Theorem~\ref{thm_Uniqueness_groundstate} that this eigenvalue is simple and the corresponding eigenfunction can be chosen strictly positive on $\mathbb{R}^d\setminus\Sigma$.

\begin{lem}\label{lem_Absolute_value_eigenfunction}
Let $u\in H^1(\mathbb{R}^d)$ be a real-valued eigenfunction of $\sfA_\alpha$ corresponding to $\lambda_1(\aa)$. Then also $|u|$ is an eigenfunction of $\sfA_\alpha$ corresponding to $\lambda_1(\alpha)$.
\end{lem}

\begin{proof}
From the fact that $|\nabla|u||=|\nabla u|$, cf. \cite[Theorem 6.17]{LiLo2001}, and $\tau_{\rm D}|u|=|\tau_{\rm D}u|$, we obtain
\begin{equation*}
\frac{\mathfrak{a}_\alpha[|u|]}{\Vert|u|\Vert_{L^2(\mathbb{R}^d)}^2}=\frac{\mathfrak{a}_\alpha[u]}{\Vert u\Vert_{L^2(\mathbb{R}^d)}^2}=\lambda_1(\alpha).
\end{equation*}
Since $\lambda_1(\alpha)$ is the bottom of the spectrum it can be represented by the min-max principle \cite[Theorem XIII.2]{ReedSimon} as
\begin{equation*}
\lambda_1(\alpha)=\inf\limits_{0\neq v\in H^1(\mathbb{R}^d)}\frac{\mathfrak{a}_\alpha[v]}{\Vert v\Vert^2_{L^2(\mathbb{R}^d)}}.
\end{equation*}
Since $\lambda_1(\alpha)$ is assumed to be a discrete eigenvalue, it follows from \cite[Chapter~10.2, Theorem~1]{BS87} that $|u|$ is indeed an eigenfunction of $\sfA_\alpha$ corresponding to the eigenvalue $\lambda_1(\alpha)$.
\end{proof}

\begin{lem}\label{lem_Zero_solution_one_point}
Let $\Omega\subseteq\mathbb{R}^d$ be open and connected. Assume that $u\in H^1(\Omega)$ and $\lambda\in\mathbb{R}$ satisfy
\begin{equation*}
(\nabla u,\nabla v)_{L^2(\Omega;\mathbb{C}^d)}=\lambda(u,v)_{L^2(\Omega)},\quad v\in H_0^1(\Omega).
\end{equation*}
Then $u\in\mathcal{C}^\infty(\Omega)$ and if $u\geq 0$ and $u(x_0)=0$ for some $x_0\in\Omega$, then $u\equiv 0$.
\end{lem}

\begin{proof}
The interior regularity $u\in\mathcal{C}^\infty(\Omega)$ is well known; cf. \cite[§6.3. Theorem~3]{Ev2010}. Assume now $u\geq 0$ and $u(x_0)=0$ for some $x_0\in\Omega$. Since $\Omega$ is connected, for every $x\in\Omega$ there exists a path $\gamma$ connecting $x$ and $x_0$. Since $\Omega$ is also open, there even exists some open and bounded $U$ with $\gamma\subseteq U\subseteq\overline{U}\subseteq\Omega$. Then it follows from Harnack's inequality \cite[Corollary 8.21]{GiTr2001}, that
\begin{equation*}
\sup\limits_{y\in U}u(y)\leq C\inf\limits_{y\in U}u(y),
\end{equation*}
for some constant $C>0$. Since $u(x_0)=0$, the right and hence also the left hand side of this inequality vanishes. Therefore, $u|_U=0$ and in particular $u(x)=0$. Since $x\in\Omega$ was arbitrary, we conclude $u\equiv 0$.
\end{proof}

\begin{lem}\label{lem_Positive_eigenfunction}
Let $u\in H^1(\mathbb{R}^d)$ be a real-valued eigenfunction of $\sfA_\alpha$ corresponding to $\lambda_1(\aa)$. Then $u\in\mathcal{C}^\infty(\mathbb{R}^d\setminus\Sigma)$ is either strictly positive or strictly negative on $\mathbb{R}^d\setminus\Sigma$.
\end{lem}

\begin{proof}
From Lemma \ref{lem_Zero_solution_one_point} we conclude $u\in\mathcal{C}^\infty(\mathbb{R}^d\setminus\Sigma)$. In order to show that $u$ has no zeros in $\mathbb{R}^d\setminus\Sigma$, we assume the converse, i.e. that $u(x_0)=0$ for some $x_0\in\mathbb{R}^d\setminus\Sigma$. It is clear that $\Sigma$ cuts the whole space $\mathbb{R}^d$ into the two domains
\begin{align*}
\Omega_+&\coloneqq\Set{(x,x_d)\in\mathbb{R}^{d-1}\times\mathbb{R} | x_d>\xi(x)},\\
\Omega_-&\coloneqq\Set{(x,x_d)\in\mathbb{R}^{d-1}\times\mathbb{R} | x_d<\xi(x)}.
\end{align*}
We will assume without loss of generality that $x_0\in\Omega_+$. Since, by Lemma \ref{lem_Absolute_value_eigenfunction}, $|u|$ is also an eigenfunction corresponding to $\lambda_1(\aa)$, we have
\begin{equation*}
(\nabla|u|,\nabla v)_{L^2(\Omega_+;{\dC^d})}=\lambda_1(\aa)(|u|,v)_{L^2(\Omega_+)},\quad v\in H_0^1(\Omega_+),
\end{equation*}
and Lemma \ref{lem_Zero_solution_one_point} implies $u|_{\Omega_+}\equiv 0$. In particular, we have $\tau_{\rm D}u=0$ and the eigenvalue equation for $u$ reduces to
\begin{equation*}
(\nabla u,\nabla v)_{L^2(\Omega_-;\mathbb{C}^d)}=\lambda_1(\aa)(u,v)_{L^2(\Omega_-)},\quad v\in H^1(\mathbb{R}^d).
\end{equation*}
Since $\lambda_1(\alpha)$ is a discrete eigenvalue, it is negative by Lemma~\ref{lemmilein}, and consequently choosing $v=u$, we conclude $u|_{\Omega_-}\equiv 0$. But this is a contradiction to the fact that $u$ is a (nonzero) eigenfunction; hence $u$ has no zeros in $\mathbb{R}^d\setminus\Sigma$.

\medskip

Since we already know that $u\in\mathcal{C}^\infty(\mathbb{R}^d\setminus\Sigma)$ has no zeros in $\mathbb{R}^d\setminus\Sigma$, it has to be either strictly positive or strictly negative on each of the domains $\Omega_\pm$. However, a priori the signs of $u$ may not coincide. If, e.g.
\begin{equation*}
u|_{\Omega_+}>0\quad\text{and}\quad u|_{\Omega_-}<0,
\end{equation*}
then $\tau_{\rm D}u=0$ and the eigenvalue equation for $u$ reduces to
\begin{equation*}
(\nabla u,\nabla v)_{L^2(\mathbb{R}^d;{\dC^d})}=\lambda_1(\aa)(u,v)_{L^2(\mathbb{R}^d)},\quad v\in H^1(\mathbb{R}^d).
\end{equation*}
Choosing $v=u$ we again conclude $u\equiv 0$ by the negativity of $\lambda_1(\alpha)$; a contradiction as $u$ is a (nonzero) eigenfunction.
\end{proof}

\begin{thm}\label{thm_Uniqueness_groundstate}
If the bottom \eqref{eq:lowest} of the spectrum of $\sfA_\alpha$ is a discrete eigenvalue, then it is simple and the corresponding eigenfunction can be chosen strictly positive on $\mathbb{R}^d\setminus\Sigma$.
\end{thm}

\begin{proof}
Note that there exists a real-valued basis of the eigenspace corresponding to $\lambda_1(\aa)$ since 
for every eigenfunction the complex conjugate is also an 
eigenfunction. Now consider two orthogonal real-valued eigenfunctions $u_1$ and $u_2$. According to 
Lemma \ref{lem_Positive_eigenfunction} each eigenfunction is either strictly positive or strictly negative on $\mathbb{R}^d\setminus\Sigma$. But this is a contradiction to the orthogonality 
condition
\begin{equation*}
\int_{\mathbb{R}^d}u_1\,u_2\,\dd x=0.
\end{equation*}
Hence, the eigenspace is one-dimensional and thus $\lambda_1(\aa)$ is a simple eigenvalue.
\end{proof}

\section{The Birman-Schwinger principle and an optimization result for $\delta$-potentials on a hyperplane}\label{sec_BS}

In this section we assume that the support of the $\delta$-potential is a hyperplane and we shall therefore identify $\Sigma=\mathbb{R}^{d-1}\times\{0\}\cong\mathbb{R}^{d-1}$. Moreover, as in \eqref{Eq_alpha}, everywhere in this section we consider a real-valued function $\alpha\in L^p(\mathbb{R}^{d-1})+L^\infty(\mathbb{R}^{d-1})$ with $1<p<\infty$ if $d=2$ and $d-1\leq p<\infty$ if $d\geq 3$. Later we shall also assume that there exists some $\alpha_0\in\mathbb{R}$ such that
\begin{equation}\label{Eq_Decay_property_relativistic}
\Set{x\in\mathbb{R}^{d-1} | \vert\alpha(x)-\alpha_0\vert>\varepsilon}\text{ has finite measure for every }\varepsilon>0.
\end{equation}
We first discuss the Birman-Schwinger principle for the operator $\sfA_\aa$ in this special situation, by means of which the spectral problem can be reduced to the spectral analysis of a relativistic Schrödinger operator in $L^2(\dR^{d-1})$. As an application and illustration we prove an optimization result for the bottom of the spectrum of $\sfA_\alpha$ in Theorem \ref{thm:main2}.

\subsection{The Birman-Schwinger principle for $\delta$-potentials supported on a hyperplane}

For every $\lambda<0$ we consider the form
\begin{equation}\label{Eq_Relativistic_form}
\begin{split}
\mathfrak{d}_{\alpha,\lambda}[\phi,\psi]&\coloneqq 2\big((-\Delta-\lambda)^{\frac{1}{4}}\phi,(-\Delta-\lambda)^{\frac{1}{4}}\psi\big)_{L^2(\mathbb{R}^{d-1})}-\int_{\mathbb{R}^{d-1}}\alpha\,\phi\,{\overline{\psi}}\,\dd x, \\
\dom\mathfrak{d}_{\alpha,\lambda}&\coloneqq H^{\frac{1}{2}}(\mathbb{R}^{d-1}).
\end{split}
\end{equation}
It follows from Lemma~\ref{lem_Malpha_boundedness}, that for every $\varepsilon>0$ there exists a $c_\varepsilon>0$ such that
\begin{equation}\label{Eq_Potential_boundedness_hyperplane}
\big\Vert|\alpha|^{\frac{1}{2}}\phi\big\Vert^2_{L^2(\mathbb{R}^{d-1})}\leq\varepsilon^2\Vert\phi\Vert^2_{H^{\frac{1}{2}}(\mathbb{R}^{d-1})}+c_\varepsilon^2\Vert\phi\Vert^2_{L^2(\mathbb{R}^{d-1})},\quad\phi\in H^{\frac{1}{2}}(\mathbb{R}^{d-1}).
\end{equation}
Using this inequality it follows (see the proof of Proposition \ref{prop_Closed_form}) that $\mathfrak{d}_{\alpha,\lambda}$ is a densely defined, symmetric, semibounded and closed form in $L^2(\mathbb{R}^{d-1})$. We denote the corresponding self-adjoint operator in $L^2(\mathbb{R}^{d-1})$ by $\sfD_{\alpha,\lambda}$. It turns out in Proposition~\ref{prop_Equivalence_of_eigenvalues} below that the eigenvalue $0$ of this {\it relativistic Schrödinger operator} is linked to the eigenvalue $\lambda$ of the Schrödinger operator $\sfA_\alpha$.

\medskip

We first formulate and prove a preparatory lemma; here we shall denote the extension of the 
$L^2(\mathbb{R}^{d-1})$ scalar product onto the dual pair $H^{-\frac{1}{2}}(\mathbb{R}^{d-1})\times H^{\frac{1}{2}}(\mathbb{R}^{d-1})$ by
$\langle\,\cdot\,,\,\cdot\,\rangle_{H^{-\frac{1}{2}}(\mathbb{R}^{d-1})\times H^{\frac{1}{2}}(\mathbb{R}^{d-1})}$.

\begin{lem}\label{lem_gamma_field}
For every $\lambda<0$ there exists a unique bounded linear operator $\gamma(\lambda):H^{-\frac{1}{2}}(\mathbb{R}^{d-1})\rightarrow H^1(\mathbb{R}^d)$ such that the identity
\begin{equation}\label{Eq_gamma_field_1}
\big(\nabla\gamma(\lambda)\phi,\nabla v\big)_{L^2(\mathbb{R}^d;{\mathbb{C}^d})}-\lambda\,\big(\gamma(\lambda)\phi,v\big)_{L^2(\mathbb{R}^d)}=\langle\phi,\tau_{\rm D}v\rangle_{H^{-\frac{1}{2}}(\mathbb{R}^{d-1})\times H^{\frac{1}{2}}(\mathbb{R}^{d-1})}
\end{equation}
holds for all $\phi\in H^{-\frac{1}{2}}(\mathbb{R}^{d-1})$ and $v\in H^1(\mathbb{R}^d)$. Moreover, the trace of $\gamma(\lambda)$ is given by
\begin{equation}\label{Eq_gamma_field_2}
\tau_{\rm D}\gamma(\lambda)=\frac{1}{2}(-\Delta-\lambda)^{-\frac{1}{2}},
\end{equation}
and acts as a bounded linear operator from $H^{-\frac{1}{2}}(\mathbb{R}^{d-1})$ to $H^{\frac{1}{2}}(\mathbb{R}^{d-1})$.
\end{lem}

\begin{proof}
Let $\cF_d$ and $\cF_{d-1}$ be the unitary Fourier transforms in $L^2(\mathbb{R}^d)$ and $L^2(\mathbb{R}^{d-1})$, respectively, and consider Schwartz functions $\phi\in\mathcal{S}(\mathbb{R}^{d-1})$. We first define the operator $\gamma(\lambda)$ in Fourier space as
\begin{equation}\label{Eq_gamma_field}
(\cF_d\gamma(\lambda)\phi)(\tilde{k})\coloneqq\frac{(\cF_{d-1}\phi)(k)}{\sqrt{2\pi}(|\tilde{k}|^2-\lambda)},\quad\tilde{k}=(k,k_d)\in\mathbb{R}^{d-1}\times\mathbb{R}.
\end{equation}
As $\lambda<0$ and $\cF_{d-1}\phi\in\mathcal{S}(\mathbb{R}^{d-1})$, this is a well defined function in $L^2(\mathbb{R}^d)$. The fact, that $\gamma(\lambda)$ is bounded from $H^{-\frac{1}{2}}(\mathbb{R}^{d-1})$ to $H^1(\mathbb{R}^d)$ follows from the estimate
\begin{equation*}
\begin{split}
\Vert\gamma(\lambda)\phi\Vert_{H^1(\mathbb{R}^d)}^2&=\frac{1}{2\pi}\int_{\mathbb{R}^d}(1+|\tilde{k}|^2)\frac{|(\cF_{d-1}\phi)(k)|^2}{(|\tilde{k}|^2-\lambda)^2}\dd\tilde{k} \\
&=\frac{1}{2\pi}\int_{\mathbb{R}^{d-1}}\int_\mathbb{R}\frac{1+|k|^2+k_d^2}{(|k|^2+k_d^2-\lambda)^2}\dd k_d|(\cF_{d-1}\phi)(k)|^2\dd k \\
&=\frac{1}{4}\int_{\mathbb{R}^{d-1}}\frac{2|k|^2+1-\lambda}{(|k|^2-\lambda)^{\frac{3}{2}}}|(\cF_{d-1}\phi)(k)|^2\dd k \\
&\leq\frac{c(\lambda)}{4}\Vert\phi\Vert_{H^{-\frac{1}{2}}(\mathbb{R}^{d-1})}^2,
\end{split}
\end{equation*}
where $c(\lambda)$ denotes the maximum of the function $k\mapsto\frac{(2|k|^2+1-\lambda)(|k|^2+1)^{1/2}}{(|k|^2-\lambda)^{3/2}}$. Since $\mathcal{S}(\mathbb{R}^{d-1})$ is dense in $H^{-\frac{1}{2}}(\mathbb{R}^{d-1})$ the operator $\gamma(\lambda)$ can be extended by continuity onto $H^{-\frac{1}{2}}(\mathbb{R}^{d-1})$.

\medskip

In order to prove the identity \eqref{Eq_gamma_field_1} for Schwartz functions $\phi\in\mathcal{S}(\mathbb{R}^{d-1})$ and $v\in\mathcal{S}(\mathbb{R}^d)$, we use the Fourier representation
\begin{equation}\label{Eq_nabla_in_Fourierspace}
(\cF_d\nabla v)(\tilde{k})=i\tilde{k}(\cF_dv)(\tilde{k}),\quad\tilde{k}\in\mathbb{R}^d,
\end{equation}
of the gradient. For $x\in\mathbb{R}^{d-1}$ the trace can be written as
\begin{equation}\label{Eq_Trace_in_Fourierspace1}
\begin{split}
(\tau_{\rm D}v)(x)&=(\cF_d^{-1}\cF_dv)(x,0)=\frac{1}{(2\pi)^{\frac{d}{2}}}\int_{\mathbb{R}^d}e^{i\langle\tilde{k},(x,0)\rangle}(\cF_dv)(\tilde{k})\dd\tilde{k} \\
&=\frac{1}{(2\pi)^{\frac{d}{2}}}\int_{\mathbb{R}^{d-1}}e^{i\langle k,x\rangle}\int_\mathbb{R}(\cF_dv)(k,k_d)\dd k_d\dd k \\
&=\frac{1}{\sqrt{2\pi}}\cF_{d-1}^{-1}\left[\int_\mathbb{R}(\cF_dv)(\,\cdot\,,k_d)\dd k_d\right](x)
\end{split}
\end{equation}
and hence 
\begin{equation}\label{Eq_Trace_in_Fourierspace}
(\cF_{d-1}\tau_{\rm D}v)(k)=\frac{1}{\sqrt{2\pi}}\int_\mathbb{R}(\cF_dv)(k,k_d)\dd k_d,\quad k\in\mathbb{R}^{d-1}.
\end{equation}
The definition \eqref{Eq_gamma_field} of $\gamma(\lambda)$, together with 
\eqref{Eq_nabla_in_Fourierspace} and \eqref{Eq_Trace_in_Fourierspace} leads to
\begin{align*}
&\big(\nabla\gamma(\lambda)\phi,\nabla v\big)_{L^2(\mathbb{R}^d;{\dC^d})}-\lambda\big(\gamma(\lambda)\phi,v\big)_{L^2(\mathbb{R}^d)} \\
&\hspace{3cm}=\int_{\mathbb{R}^d}(|\tilde{k}|^2-\lambda)\,(\cF_d\gamma(\lambda)\phi)(\tilde{k})\,\overline{(\cF_d v)(\tilde{k})}\,\dd\tilde{k} \\
&\hspace{3cm}=\frac{1}{\sqrt{2\pi}}\int_{\mathbb{R}^d}
(\cF_{d-1}\phi)(k)\,\overline{(\cF_dv)(k,k_d)}\,\dd k_d\dd k \\
&\hspace{3cm}=\int_{\mathbb{R}^{d-1}}(\cF_{d-1}\phi)(k)\,\overline{(\cF_{d-1}\tau_{\rm D}v)(k)}\,\dd k \\
&\hspace{3cm}=(\phi,\tau_{\rm D}v)_{L^2(\mathbb{R}^{d-1})} \\
&\hspace{3cm}=\langle\phi,\tau_{\rm D}v\rangle_{H^{-\frac{1}{2}}(\mathbb{R}^{d-1})\times H^{\frac{1}{2}}(\mathbb{R}^{d-1})},
\end{align*}
and hence \eqref{Eq_gamma_field_1} holds for $\phi\in\mathcal{S}(\mathbb{R}^{d-1})$ and $v\in\mathcal{S}(\mathbb{R}^d)$. By density and continuity 
this identity extends to all $\phi\in H^{-\frac{1}{2}}(\mathbb{R}^{d-1})$ and $v\in H^1(\mathbb{R}^d)$. Also note, that the identity \eqref{Eq_gamma_field_1} uniquely defines the operator $\gamma(\lambda)$.

\medskip

For the proof of \eqref{Eq_gamma_field_2}
note first that the identity \eqref{Eq_Trace_in_Fourierspace} and its derivation \eqref{Eq_Trace_in_Fourierspace1} remain valid for 
$v\in H^1(\mathbb{R}^d)\cap\mathcal{C}(\mathbb{R}^d)$ with $\mathcal{F}_dv\in L^1(\mathbb{R}^d)$. In particular, for $\phi\in\mathcal{S}(\mathbb{R}^{d-1})$ it is not difficult to see that $\mathcal{F}_d\gamma(\lambda)\phi\in L^1(\mathbb{R}^d)$ by its definition \eqref{Eq_gamma_field} and hence also that $\gamma(\lambda)\phi=\mathcal{F}_d^{-1}\mathcal{F}_d\gamma(\lambda)\phi$ is continuous as the inverse Fourier transform of an $L^1$-function. This means that 
from \eqref{Eq_Trace_in_Fourierspace} we get
\begin{align*}
(\cF_{d-1}\tau_{\rm D}\gamma(\lambda)\phi)(k)&=\frac{1}{\sqrt{2\pi}}\int_\mathbb{R}(\cF_d\gamma(\lambda)\phi)(k,k_d)\dd k_d\\
&=\frac{(\cF_{d-1}\phi)(k)}{2\pi}\int_\mathbb{R}\frac{\dd k_d}{|\tilde{k}|^2-\lambda}=\frac{(\cF_{d-1}\phi)(k)}{2(|k|^2-\lambda)^{\frac{1}{2}}},
\end{align*}
which is exactly equation \eqref{Eq_gamma_field_2} in Fourier space. Again, by continuity this identity also holds for every $\phi\in H^{-\frac{1}{2}}(\mathbb{R}^{d-1})$.
\end{proof}

With this lemma we now find a connection between the eigenvalue $0$ of the relativistic Schrödinger operator $\sfD_{\alpha,\lambda}$ and the eigenvalue $\lambda$ of the Schrödinger operator $\sfA_\alpha$.

\begin{prop}\label{prop_Equivalence_of_eigenvalues}
For every $\lambda<0$ the restriction of the Dirichlet trace operator
\begin{equation}\label{Eq_Equivalence_of_eigenvalues}
\tau_{\rm D}:\ker(\sfA_\alpha-\lambda)\rightarrow\ker\sfD_{\alpha,\lambda}
\end{equation}
is bijective and, in particular, $\dim\ker(\sfA_\alpha-\lambda)=\dim\ker\sfD_{\alpha,\lambda}$.
\end{prop}

\begin{proof}
In order to see that the restriction of $\tau_{\rm D}$ onto $\ker(\sfA_\alpha-\lambda)$ maps into $\ker\sfD_{\alpha,\lambda}$ consider some $u\in\ker(\sfA_\alpha-\lambda)$. By \eqref{Eq_Schroedinger_form} we have $u\in H^1(\mathbb{R}^d)$ and
\begin{equation}\label{Eq_Equivalence_of_eigenspaces_1}
(\nabla u,\nabla v)_{L^2(\mathbb{R}^d;{\dC^d})}-\lambda(u,v)_{L^2(\mathbb{R}^d)}=\big(\sgn(\alpha)|\alpha|^{\frac{1}{2}}\tau_{\rm D}u,|\alpha|^{\frac{1}{2}}\tau_{\rm D}v\big)_{L^2(\mathbb{R}^{d-1})}
\end{equation}
for all $v\in H^1(\mathbb{R}^d)$. Since $\tau_{\rm D}u\in H^{\frac{1}{2}}(\mathbb{R}^{d-1})$, we get $|\alpha|^{\frac{1}{2}}\tau_{\rm D}u\in L^2(\mathbb{R}^{d-1})$ from \eqref{Eq_Potential_boundedness_hyperplane} and hence there exist  $\psi_n\in H^{\frac{1}{2}}(\mathbb{R}^{d-1})$ such that 
\begin{equation}\label{Eq_Equivalence_of_eigenspaces_2}
{\sgn(\alpha)}|\alpha|^{\frac{1}{2}}\tau_{\rm D}u=\lim\limits_{n\rightarrow\infty}\psi_n\quad\text{in }L^2(\mathbb{R}^{d-1}).
\end{equation}
Again, by (\ref{Eq_Potential_boundedness_hyperplane}), we have $|\alpha|^{\frac{1}{2}}\psi_n\in L^2(\mathbb{R}^{d-1})$ and inserting these 
into \eqref{Eq_gamma_field_1} leads to
\begin{equation*}
\big(\nabla\gamma(\lambda)|\alpha|^{\frac{1}{2}}\psi_n,\nabla v\big)_{L^2(\mathbb{R}^d;{\dC^d})}-\lambda\big(\gamma(\lambda)|\alpha|^{\frac{1}{2}}\psi_n,v\big)_{L^2(\mathbb{R}^d)}=\big(\psi_n,|\alpha|^{\frac{1}{2}}\tau_{\rm D}v\big)_{L^2(\mathbb{R}^{d-1})}
\end{equation*}
for all $v\in H^1(\mathbb{R}^d)$.
Combining this with \eqref{Eq_Equivalence_of_eigenspaces_1} and \eqref{Eq_Equivalence_of_eigenspaces_2} implies the 
convergence
\begin{equation*}
\gamma(\lambda)|\alpha|^{\frac{1}{2}}\psi_n\rightharpoonup u\quad\text{weakly in }H^1(\mathbb{R}^d).
\end{equation*}
Applying the bounded operator $(-\Delta-\lambda)^{\frac{1}{4}}\tau_{\rm D}:H^1(\mathbb{R}^d)\rightarrow L^2(\mathbb{R}^{d-1})$ and using (\ref{Eq_gamma_field_2}) gives
\begin{equation*}
\frac{1}{2}(-\Delta-\lambda)^{-\frac{1}{4}}|\alpha|^{\frac{1}{2}}\psi_n=(-\Delta-\lambda)^{\frac{1}{4}}\tau_{\rm D}\gamma(\lambda)|\alpha|^{\frac{1}{2}}\psi_n\rightharpoonup(-\Delta-\lambda)^{\frac{1}{4}}\tau_{\rm D}u
\end{equation*}
weakly in $L^2(\mathbb{R}^{d-1})$. Hence, for every $\psi\in H^{\frac{1}{2}}(\mathbb{R}^{d-1})$ we get
\begin{align*}
\mathfrak{d}_{\alpha,\lambda}[\tau_{\rm D}u,\psi]&=\lim\limits_{n\rightarrow\infty}\big((-\Delta-\lambda)^{-\frac{1}{4}}|\alpha|^{\frac{1}{2}}\psi_n,(-\Delta-\lambda)^{\frac{1}{4}}\psi\big)_{L^2(\mathbb{R}^{d-1})}-\int_{\mathbb{R}^{d-1}}\alpha\,\tau_{\rm D}u\,\overline{\psi}\,\dd x \\
&=\lim\limits_{n\rightarrow\infty}\big(\psi_n,|\alpha|^{\frac{1}{2}}\psi\big)_{L^2(\mathbb{R}^{d-1})}-\int_{\mathbb{R}^{d-1}}\alpha\,\tau_{\rm D}u\,\overline{\psi}\,\dd x=0,
\end{align*}
where \eqref{Eq_Equivalence_of_eigenspaces_2} was used in the last step. Thus, we conclude $\tau_{\rm D}u\in\ker\sfD_{\alpha,\lambda}$.

\medskip

Next we show that \eqref{Eq_Equivalence_of_eigenvalues} is injective. In fact, assume that $\tau_{\rm D}u=0$ for some $u\in\ker(\sfA_\alpha-\lambda)$. Then 
\eqref{Eq_Schroedinger_form} leads to
\begin{equation*}
(\nabla u,\nabla v)_{L^2(\mathbb{R}^d;{\dC^d})}=\lambda(u,v)_{L^2(\mathbb{R}^{d-1})},\quad v\in H^1(\mathbb{R}^d).
\end{equation*}
Since $\lambda<0$ we can choose $v=u$ and conclude $u=0$.

\medskip

For the surjectivity of \eqref{Eq_Equivalence_of_eigenvalues} let $\phi\in\ker\sfD_{\alpha,\lambda}$. By \eqref{Eq_Relativistic_form} we have 
$\phi\in H^{\frac{1}{2}}(\mathbb{R}^{d-1})$ and
\begin{equation}\label{Eq_Equivalence_of_eigenspaces_4}
2\big((-\Delta-\lambda)^{\frac{1}{4}}\phi,(-\Delta-\lambda)^{\frac{1}{4}}\psi\big)_{L^2(\mathbb{R}^{d-1})}=\int_{\mathbb{R}^{d-1}}\alpha\,\phi\,\overline{\psi}\,\dd x,\quad\psi\in H^{\frac{1}{2}}(\mathbb{R}^{d-1}).
\end{equation}
Now define $u_\phi\coloneqq 2\gamma(\lambda)(-\Delta-\lambda)^{\frac{1}{2}}\phi$. Then $\tau_{\rm D}u_\phi=\phi$ by \eqref{Eq_gamma_field_2} and using  \eqref{Eq_gamma_field_1} with $\phi$ replaced by $2(-\Delta-\lambda)^{\frac{1}{2}}\phi$, gives for any $v\in H^1(\dR^d)$
\begin{align*}
\big(\nabla u_\phi,\nabla v\big)_{L^2(\mathbb{R}^d;{\mathbb{C}^d})}-\lambda\,\big(u_\phi,v\big)_{L^2(\mathbb{R}^d)}&=2\langle(-\Delta-\lambda)^{\frac{1}{2}}\phi,\tau_{\rm D}v\rangle_{H^{-\frac{1}{2}}(\mathbb{R}^{d-1})\times H^{\frac{1}{2}}(\mathbb{R}^{d-1})} \\
&=\int_{\mathbb{R}^{d-1}}\alpha\,\phi\,\overline{\tau_{\rm D}v}\,\dd x \\
&=\int_{\mathbb{R}^{d-1}}\alpha\,\tau_{\rm D}u_\phi\,\overline{\tau_{\rm D}v}\,\dd x,
\end{align*}
where in the second step we used \eqref{Eq_Equivalence_of_eigenspaces_4} with $\psi=\tau_{\rm D}v$. Summing up, for $\phi\in\ker\sfD_{\alpha,\lambda}$ we found  $u_\phi\in\ker(\sfA_\alpha-\lambda)$ such that $\tau_{\rm D}u_\phi=\phi$, that is, the mapping \eqref{Eq_Equivalence_of_eigenvalues} is surjective.
\end{proof}

Next we analyse how the bottom of the spectrum $\sigma(\sfD_{\alpha,\lambda})$ behaves as a function of $\lambda<0$.

\begin{lem}\label{lem_Eigenvalue_function}
For $\lambda<0$ the mapping
\begin{equation}\label{Eq_Eigenvalue_function}
\lambda\mapsto \mu_\aa(\lambda)\coloneqq\inf\s(\sfD_{\aa,\lm}) = 
\inf\limits_{0\neq\phi\in H^{1/2}(\dR^{d-1})}\frac{\mathfrak{d}_{\alpha,\lambda}[\phi]}{\Vert\phi\Vert^2_{L^2(\mathbb{R}^{d-1})}}
\end{equation}
is nonincreasing, continuous and admits the limit
\begin{equation}\label{Eq_Eigenvalue_function_limit}
\lim\limits_{\lambda\rightarrow-\infty}\mu_\aa(\lambda)=\infty.
\end{equation}
\end{lem}

\begin{proof}
With the help of the Fourier transform in $L^2(\mathbb{R}^{d-1})$ we see that the form $\mathfrak{d}_{\alpha,\lambda}$ admits the representation 
\begin{equation}\label{Eq_Properties_gn_1}
\mathfrak{d}_{\alpha,\lambda}[\phi]=2\int_{\mathbb{R}^{d-1}}(|k|^2-\lambda)^{\frac{1}{2}}|(\cF_{d-1}\phi)(k)|^2\dd k-\int_{\mathbb{R}^{d-1}}\alpha\,|\phi|^2\,\dd x,\quad\phi\in H^{\frac{1}{2}}(\mathbb{R}^{d-1}),
\end{equation}
which shows that $\mathfrak{d}_{\alpha,\lambda}[\phi]$ is nonincreasing in $\lambda$. Hence the same is true for $\mu_\alpha$ in \eqref{Eq_Eigenvalue_function}.

\medskip

For the continuity of $\mu_\alpha$ consider $\lambda_1\leq\lambda_2<0$. Then for every $\phi\in H^{\frac{1}{2}}(\mathbb{R}^{d-1})$ we can estimate the difference
\begin{align*}
\mathfrak{d}_{\alpha,\lambda_1}[\phi]-\mathfrak{d}_{\alpha,\lambda_2}[\phi]&=2\int_{\mathbb{R}^{d-1}}\big((|k|^2-\lambda_1)^{\frac{1}{2}}-(|k|^2-\lambda_2)^{\frac{1}{2}}\big)|(\cF_{d-1}\phi)(k)|^2\dd k \\
&\leq 2\big(\sqrt{-\lambda_1}-\sqrt{-\lambda_2}\big)\,\Vert\phi\Vert^2_{L^2(\mathbb{R}^{d-1})},
\end{align*}
and via \eqref{Eq_Eigenvalue_function} we also conclude 
\begin{equation*}
\mu_\aa(\lm_1)-\mu_\aa(\lm_2)\leq 2\big(\sqrt{-\lambda_1}-\sqrt{-\lambda_2}\big),
\end{equation*}
which proves the continuity of $\lambda\mapsto\mu_\aa(\lm)$.

\medskip

It remains to verify \eqref{Eq_Eigenvalue_function_limit}. For this we use the estimate
\begin{equation*}
\Big|\int_{\mathbb{R}^{d-1}}\alpha\,|\phi|^2\dd x\Big|\leq\Vert\phi\Vert^2_{H^{\frac{1}{2}}(\mathbb{R}^{d-1})}+c_1^2\Vert\phi\Vert^2_{L^2(\mathbb{R}^{d-1})},\quad\phi\in H^{\frac{1}{2}}(\mathbb{R}^{d-1}),
\end{equation*}
from \eqref{Eq_Potential_boundedness_hyperplane}. Plugging this in \eqref{Eq_Properties_gn_1} gives
\begin{align*}
\mathfrak{d}_{\alpha,\lambda}[\phi]&\geq\int_{\mathbb{R}^{d-1}}\big(2(|k|^2-\lambda)^{\frac{1}{2}}-(1+|k|^2)^{\frac{1}{2}}\big)|(\cF_{d-1}\phi)(k)|^2\dd k-c_1^2\Vert\phi\Vert^2_{L^2(\mathbb{R}^{d-1})} \\
&\geq(c(\lambda)-c_1^2)\Vert\phi\Vert^2_{L^2(\mathbb{R}^{d-1})},
\end{align*}
where $c(\lambda)\in\mathbb{R}$ is the minimum of $k\mapsto 2(|k|^2-\lambda)^{\frac{1}{2}}-(1+|k|^2)^{\frac{1}{2}}$. From \eqref{Eq_Eigenvalue_function} we then conclude
\begin{equation*}
\mu_\aa(\lm)\geq c(\lambda)-c_1^2\overset{\lambda\rightarrow-\infty}{\longrightarrow}\infty.\qedhere
\end{equation*}
\end{proof}

Next, we compute the essential spectrum of $\sfD_{\alpha,\lambda}$ under the additional assumption that $\alpha$ satisfies the decay condition \eqref{Eq_Decay_property_relativistic}.

\begin{prop}\label{prop_ess}
Assume that $\alpha$ satisfies \eqref{Eq_Decay_property_relativistic} with some $\alpha_0\in\mathbb R$. Then for every $\lambda<0$ the essential spectrum of $\sfD_{\alpha,\lambda}$ is given by
\begin{equation}\label{Eq_Essential_spectrum_D}
\sess(\sfD_{\alpha,\lambda})=\big[2\sqrt{-\lambda}-\alpha_0,\infty\big).
\end{equation}
Furthermore, the mapping $\lambda\mapsto\mu_\alpha(\lambda)$ from \eqref{Eq_Eigenvalue_function} is strictly decreasing on $(-\infty,0)$.
\end{prop}

\begin{proof}
It is clear that for the special case of a constant $\alpha(x)=\alpha_0\in\mathbb R$ the relativistic Schrödinger operator is given by $\sfD_{\alpha_0,\lambda}=2(-\Delta-\lambda)^{\frac{1}{2}}-\alpha_0$ with $\dom \sfD_{\alpha_0,\lambda}=H^1(\mathbb{R}^{d-1})$. Hence 
we have
\begin{equation}\label{juhu}
\sigma(\sfD_{\alpha_0,\lambda})=\sess(\sfD_{\aa_0,\lm})=\big[2\sqrt{-\lambda}-\alpha_0,\infty\big).
\end{equation}

For nonconstant $\alpha$ we define $\alpha_1(x)\coloneqq\alpha(x)-\alpha_0$. Then $\{x\in\mathbb{R}^{d-1} \,|\, \vert\alpha_1(x)\vert>\varepsilon\}$ has finite measure for every $\varepsilon>0$ by the decay property \eqref{Eq_Decay_property_relativistic}. To prove \eqref{Eq_Essential_spectrum_D} we proceed in the same way as in Step 3 of the proof of Theorem~\ref{prop_ess} and check that for some $\mu<\inf(\sigma(\sfD_{\alpha_0,\lambda})\cup\sigma(\sfD_{\alpha,\lambda}))$ the resolvent difference
\begin{equation*}
\sfW\coloneqq(\sfD_{\alpha_0,\lambda}-\mu)^{-1}-(\sfD_{\alpha,\lambda}-\mu)^{-1}
\end{equation*}
is a compact operator in $L^2(\mathbb{R}^{d-1})$. For this let $\phi,\psi\in L^2(\mathbb{R}^{d-1})$ and set
\begin{equation*}
\phi_\mu\coloneqq(\sfD_{\alpha_0,\lambda}-\mu)^{-1}\phi\quad\text{and}\quad\psi_\mu\coloneqq(\sfD_{\alpha,\lambda}-\mu)^{-1}\psi.
\end{equation*}
In the same way as in \eqref{Eq_W_reduction} one verifies
\begin{align*}
(\sfW\phi,\psi)_{L^2(\mathbb{R}^{d-1})}&=\big(\phi_\mu,\sfD_{\alpha,\lambda}\psi_\mu\big)_{L^2(\mathbb{R}^{d-1})}-\big(\sfD_{\alpha_0,\lambda}\phi_\mu,\psi_\mu\big)_{L^2(\mathbb{R}^{d-1})}\\
&=-\int_{\mathbb{R}^{d-1}}\alpha_1\,\phi_\mu\overline{\psi_\mu}\dd x \\
&=(\sfT_1\phi,\sfT_2\psi)_{L^2(\mathbb{R}^{d-1})},
\end{align*}
where 
\begin{equation*}
\sfT_1\coloneqq|\alpha_1|^{\frac{1}{2}}(\sfD_{\alpha_0,\lambda}-\mu)^{-1}\quad\text{and}\quad\sfT_2\coloneqq-\sgn(\alpha_1)|\alpha_1|^{\frac{1}{2}}(\sfD_{\alpha,\lambda}-\mu)^{-1}.
\end{equation*}
As $(\sfD_{\alpha_0,\lambda}-\mu)^{-1}$ and $(\sfD_{\alpha,\lambda}-\mu)^{-1}$ are bounded operators from $L^2(\mathbb{R}^{d-1})$ into $H^{\frac{1}{2}}(\mathbb{R}^{d-1})$ it follows from Proposition \ref{prop_Malpha_compactness} that both $\sfT_1$ and $\sfT_2$ are compact operators in $L^2(\mathbb{R}^{d-1})$. Thus the resolvent difference $\sfW=\sfT_2^*\sfT_1$ is compact as well, which implies $\sess(\sfD_{\alpha_0,\lambda})=\sess(\sfD_{\alpha,\lambda})$ and \eqref{Eq_Essential_spectrum_D} follows from \eqref{juhu}.

\medskip

For the proof of the strict monotonicity of $\lambda\mapsto\mu_\alpha(\lambda)$, let $\lm_1<\lm_2<0$. Then 
\begin{equation}\label{eq:mu_aa_est}
\mu_\aa(\lm_j)\le 2\sqrt{-\lm_j}-\aa_0, \qquad j=1,2,
\end{equation}
by \eqref{Eq_Essential_spectrum_D}. If $\mu_\aa(\lm_1)=2\sqrt{-\lm_1}-\aa_0$ we conclude from $\mu_\aa(\lm_2)\le 2\sqrt{-\lm_2}-\aa_0$ that $\mu_\aa(\lm_2)<\mu_\aa(\lm_1)$. If $\mu_\aa(\lm_1)<2\sqrt{-\lm_1}-\aa_0$ we know from \eqref{Eq_Essential_spectrum_D} that $\mu_\alpha(\lambda_1)$ is a discrete eigenvalue of $\sfD_{\aa,\lm_1}$ and hence there is a corresponding eigenfunction $\phi\in\dom\sfD_{\aa,\lm_1}\subset H^{\frac12}(\dR^{d-1})$. Since, in particular, $\phi\neq 0$ we conclude from \eqref{Eq_Properties_gn_1} that $\lambda\mapsto\mathfrak{d}_{\alpha,\lambda}[\phi]$ is strictly decreasing, and hence
\begin{equation*}
\mu_\alpha(\lambda_1)=\frac{\mathfrak{d}_{\alpha,\lambda_1}[\phi]}{\Vert\phi\Vert^2_{L^2(\mathbb{R}^{d-1})}}>\frac{\mathfrak{d}_{\alpha,\lambda_2}[\phi]}{\Vert\phi\Vert^2_{L^2(\mathbb{R}^{d-1})}}\geq\mu_\alpha(\lambda_2).\qedhere
\end{equation*}
\end{proof}

\begin{lem}\label{lem_Lowest_eigenvalue}
Assume that $\alpha$ satisfies \eqref{Eq_Decay_property_relativistic} with some $\alpha_0\in\mathbb R$.
For the lowest spectral point $\lambda_1(\alpha)$ of $\sfA_\alpha$ in \eqref{eq:lowest} and the lowest spectral point $\mu_\alpha(\lambda)$ of $\sfD_{\alpha,\lambda}$ in \eqref{Eq_Eigenvalue_function} the following are equivalent:

\begin{itemize}
\item[{\rm (i)}] $\lambda_1(\alpha)\in\sigma_{\rm d}(\sfA_\alpha)$

\item[{\rm (ii)}] $\mu_\alpha$ admits a zero strictly below $\left\{\begin{array}{ll} -\frac{\alpha_0^2}{4}, & \text{if }\alpha_0\geq 0 \\ 0, & \text{if }\alpha_0\leq 0. \end{array}\right.$ 
\end{itemize}

In this situation the zero of $\mu_\alpha$ coincides with $\lambda_1(\alpha)$.
\end{lem}

\begin{proof}
For an easier notation we write $\lambda_1\coloneqq\lambda_1(\alpha)$. For the implication ${\rm (i)}\Rightarrow{\rm (ii)}$ let $\lambda_1\in\sigma_{\rm d}(\sfA_\alpha)$ and note that due to the explicit form of the essential spectrum \eqref{Eq_Essential_spectrum} we have
\begin{equation}\label{Eq_lambda1_estimate}
\lambda_1<\left\{\begin{array}{ll} -\frac{\alpha_0^2}{4}, & \text{if }\alpha_0\geq 0, \\ 0, & \text{if }\alpha_0\leq 0. \end{array}\right.
\end{equation}
It follows from Proposition \ref{prop_Equivalence_of_eigenvalues} that zero is an eigenvalue of $\sfD_{\alpha,\lambda_1}$. Assume now $\mu_\alpha(\lambda_1)\neq 0$.

\begin{itemize}
\item The case $\mu_\alpha(\lambda_1)=\inf\sigma(\sfD_{\alpha,\lambda_1})>0$ is a contradiction to the fact that zero is an eigenvalue of $\sfD_{\aa,\lm_1}$.

\item If $\mu_\alpha(\lambda_1)<0$, then $\mu_\alpha(\tilde{\lambda})=0$ for some $\tilde{\lambda}<\lambda_1$ by Lemma \ref{lem_Eigenvalue_function}. Also note, that 
\begin{equation*}
\inf\sess(\sfD_{\alpha,\tilde{\lambda}})= 2\sqrt{-\tilde{\lambda}}-\alpha_0\geq 2\sqrt{-\lambda_1}-\alpha_0>0
\end{equation*}
by Proposition \ref{prop_ess} and the estimate \eqref{Eq_lambda1_estimate}. But then the bottom of the spectrum
\begin{equation*}
0=\mu_\alpha(\tilde{\lambda})=\inf\sigma(\sfD_{\alpha,\tilde{\lambda}})
\end{equation*}
is a point in the discrete spectrum and hence an eigenvalue of $\sfD_{\alpha,\tilde{\lambda}}$. Consequently, Proposition \ref{prop_Equivalence_of_eigenvalues} implies that $\tilde{\lambda}<\lambda_1$ is an eigenvalue of $\sfA_\alpha$; a contradiction as $\lambda_1$ is the smallest spectral point of $\sfA_\alpha$.
\end{itemize}

Hence our assumption is wrong and we conclude $\mu_\alpha(\lambda_1)=0$. Due to the strict monotonicity in Proposition \ref{prop_ess}, this is also the only zero of $\mu_\alpha$.

For the implication ${\rm(ii)}\Rightarrow{\rm(i)}$ assume that $\mu_\alpha$ admits a zero
\begin{equation}\label{Eq_lambdatilde_estimate}
\tilde{\lambda}<\left\{\begin{array}{ll} -\frac{\alpha_0^2}{4}, & \text{if }\alpha_0\geq 0, \\ 0, & \text{if }\alpha_0\leq 0, \end{array}\right.
\end{equation}
that is, $0=\mu_\alpha(\tilde{\lambda})=\inf\sigma(\sfD_{\alpha,\tilde{\lambda}})$. Since $2\sqrt{-\tilde{\lambda}}-\alpha_0>0$ by \eqref{Eq_lambdatilde_estimate} we conclude from \eqref{Eq_Essential_spectrum_D} that zero belongs to the discrete spectrum of $\sfD_{\alpha,\tilde{\lambda}}$, and hence Proposition~\ref{prop_Equivalence_of_eigenvalues} implies that $\tilde{\lambda}$ is an eigenvalue of $\sfA_\alpha$. Hence, also the bottom of the spectrum
\begin{equation*}
\lambda_1=\inf\sigma(\sfA_\alpha)\leq\tilde{\lambda}<\left\{\begin{array}{ll} -\frac{\alpha_0^2}{4}, & \text{if }\alpha_0\geq 0, \\ 0, & \text{if }\alpha_0\leq 0, \end{array}\right.
\end{equation*}
belongs to the discrete spectrum of $\sfA_\alpha$ by \eqref{Eq_Essential_spectrum}.
\end{proof}

\subsection{Optimization of $\lambda_1(\aa)$ and the symmetric decreasing rearrangement.}\label{sec_opti}

In this subsection we prove an optimization result for the bottom of the spectrum of $\sfA_\alpha$, which will be formulated in terms of the so-called symmetric decreasing rearrangement of the positive part of the function $\alpha_1(x)\coloneqq\alpha(x)-\alpha_0$, with $\alpha_0\in\mathbb{R}$ from \eqref{Eq_Decay_property_relativistic}. We first briefly recall the definition and some basic properties of the symmetric decreasing rearrangement and formulate our main result in Theorem~\ref{thm:main2} below. Further details on symmetric decreasing rearrangements can be found  in the monographs~\cite{B,LiLo2001}.

\medskip

Let $\cA\subseteq\mathbb{R}^{d-1}$, $d\geq 2$, be a measurable set of finite volume. Then its \emph{symmetric rearrangement} $\cA^*$ is defined as the open ball centered at the origin and having the same volume. Let $u\colon\dR^{d-1}\arr\dR$ be a nonnegative measurable function, that vanishes at infinity in the sense that
\begin{equation}\label{Eq_Decay_property_u}
\Set{x\in\mathbb{R}^{d-1} | u(x)>t}\text{ has finite measure for every }t>0.
\end{equation}
We define the \emph{symmetric decreasing rearrangement} $u^*$ of $u$ by symmetrizing its level sets as
\begin{equation}\label{Eq_Symmetric_decreasing_rearrangement}
u^*(x)\coloneqq\int_0^\infty\chi_{\{u>t\}^*}(x)\,\dd t.
\end{equation}
Here $\chi_\cA\colon\dR^{d-1}\arr\dR$ denotes the characteristic function. The rearrangement $u^*$ has a number of straightforward properties, which will be needed below in the proofs of Theorem~\ref{thm:main2} and Lemma~\ref{lem_Properties_rearrangement}; cf. \cite[Section 3.3 (iv) and Theorem 3.4]{LiLo2001}.

\begin{lem}\label{lem:rearr}
Let $u,v\colon\dR^{d-1}\arr\dR$ be nonnegative measurable functions satisfying \eqref{Eq_Decay_property_u}. Then the following holds:

\begin{myenum}
\item $u^*$ is nonnegative;
\item $u^*$ is radially symmetric and nonincreasing;
\item $u$ and $u^*$ are equi-measurable, i.e.,
\begin{equation*}
\big|\Set{x\in\dR^{d-1} | u(x)>t}\big|=\big|\Set{x\in\dR^{d-1} | u^*(x)>t}\big|,\quad t>0;
\end{equation*}
\item $(u^*)^2=(u^2)^*$.
\item $\Vert u\Vert_{L^p(\dR^{d-1})}=\Vert u^*\Vert_{L^p(\dR^{d-1})}$,\quad $p\geq 1$\quad (Conservation of $L^p$-norm);
\item $\int_{\dR^{d-1}}u\,v\,\dd x\leq\int_{\dR^{d-1}}u^*v^*\dd x$\quad (Hardy-Littlewood inequality).
\end{myenum}
\end{lem}

Next we formulate our optimization result for the bottom of the spectrum of $\sfA_\alpha$.

\begin{thm}\label{thm:main2} 
Assume that $\alpha$ satisfies \eqref{Eq_Decay_property_relativistic} with some $\alpha_0\in\mathbb R$ and let $\alpha_1(x)\coloneqq\alpha(x)-\alpha_0$. Then we have the inequality
\begin{equation*}
\lambda_1(\alpha_0+(\alpha_1)_+^*)\leq\lambda_1(\alpha_0+\alpha_1),
\end{equation*}
where $(\alpha_1)_+^*$ is the symmetric decreasing rearrangement of the positive part $(\alpha_1)_+\coloneqq\max\{\alpha_1,0\}$ defined in \eqref{Eq_Symmetric_decreasing_rearrangement}. 
\end{thm}

\begin{cor}\label{qqq}
Let $\omg\subset\dR^{d-1}$ be a set of finite measure and $\omg^*\subset\dR^{d-1}$ be a ball with the same volume as $\omg$, and let  $\chi_{\omg}$ and $\chi_{\omg^*}$ be the characteristic functions of $\omg$ and $\omg^*$, respectively. Then for $\beta\geq 0$ we have the inequality
\begin{equation*}
\lambda_1(\beta\chi_{\omg^*})\leq\lambda_1(\beta\chi_\omg).
\end{equation*}
\end{cor}

The proof of Theorem~\ref{thm:main2} relies on the Birman-Schwinger principle for the operator $\sfA_\aa$, by means of which the problem is reduced 
to an eigenvalue inequality for the relativistic Schr\"odinger operator in $L^2(\dR^{d-1})$. The latter is proven with the help of the fact that the symmetric decreasing rearrangement decreases the kinetic energy term corresponding to the relativistic Schr\"odinger operator; cf. Lemma~\ref{lem_Properties_rearrangement}. This property of the kinetic energy can be viewed as an analogue of the P\'{o}lya-Szeg\H{o} inequality.

\begin{lem}\label{lem_Properties_rearrangement}
For every $\lambda<0$ and nonnegative $\phi\in H^{\frac{1}{2}}(\mathbb{R}^{d-1})$ the rearrangements $(\alpha_1)_+^*,\phi^*$ 
in \eqref{Eq_Symmetric_decreasing_rearrangement} and the form \eqref{Eq_Relativistic_form} satisfy
\begin{equation}\label{Eq_Relativistic_form_estimate}
\mathfrak{d}_{\alpha_0+(\alpha_1)_+^*,\lambda}[\phi^*]\leq\mathfrak{d}_{\alpha_0+\alpha_1,\lambda}[\phi].
\end{equation}
\end{lem}

\begin{proof}
First, in view of Lemma \ref{lem:rearr} (iv), (v) and (vi) we have
\begin{equation}\label{Eq_Properties_rearrangement1}
\int_{\dR^{d-1}}(\aa_0+\aa_1)\phi^2\dd x\leq\int_{\dR^{d-1}}(\aa_0+(\aa_1)_+)\phi^2\dd x\leq\int_{\dR^{d-1}}(\aa_0+(\aa_1)_+^*)(\phi^*)^2\dd x.
\end{equation}
Moreover, it is proven in \cite[Section 7.11 (5), Section 7.17 (2) and the remark afterwards]{LiLo2001} that
\begin{equation}\label{Eq_Properties_rearrangement2}
\big\Vert(-\Delta-\lambda)^{\frac{1}{4}}\phi^*\big\Vert^2_{L^2(\mathbb{R}^{d-1})}\leq\big\Vert(-\Delta-\lambda)^{\frac{1}{4}}\phi\big\Vert_{L^2(\mathbb{R}^{d-1})}.
\end{equation}
Combining \eqref{Eq_Properties_rearrangement1} and \eqref{Eq_Properties_rearrangement2} then proves the stated inequality \eqref{Eq_Relativistic_form_estimate}.
\end{proof}

\begin{proof}[Proof of Theorem \ref{thm:main2}]
Observe that by Theorem~\ref{thm:Essential_spectrum} and Lemma~\ref{lem:rearr}\,(v) the essential spectra of the Schrödinger operators $\sfA_{\alpha_0+\alpha_1}$ and $\sfA_{\alpha_0+(\alpha_1)_+^*}$ are given by
\begin{equation*}
\sess(\sfA_{\aa_0+\aa_1})=\sess(\sfA_{\aa_0+(\aa_1)_+^*})=\left\{\begin{array}{ll} [-\frac{\alpha_0^2}{4},\infty), & \text{if }\alpha_0\geq 0, \\ {[0,\infty)}, & \text{if }\alpha_0\leq 0. \end{array}\right.
\end{equation*}
We assume that $\aa_1$ is such that
\begin{equation*}
\lambda_1\coloneqq\lambda_1(\alpha_0+\alpha_1)<\left\{\begin{array}{ll} -\frac{\alpha_0^2}{4}, & \text{if }\alpha_0\geq 0, \\ 0, & \text{if }\alpha_0\leq 0, \end{array}\right.
\end{equation*}
as otherwise the statement of the theorem is clear. Then $\lambda_1\in\sigma_{\rm d}(\sfA_{\alpha_0+\alpha_1})$ and by Theorem \ref{thm_Uniqueness_groundstate} there exists a nonnegative eigenfunction $u_1\in\ker(\sfA_{\alpha_0+\alpha_1}-\lambda_1)$. By Proposition \ref{prop_Equivalence_of_eigenvalues}, we then have $\phi_1\coloneqq\tau_Du_1\in\ker\sfD_{\alpha_0+\alpha_1,\lambda_1}$ for the trace of the eigenfunction, and also $\phi_1\geq 0$ follows from $u_1\geq 0$. Lemma~\ref{lem:rearr}~(v) and Lemma~\ref{lem_Properties_rearrangement} give the estimate
\begin{equation*}
0=\frac{\mathfrak{d}_{\alpha_0+\alpha_1,\lambda_1}[\phi_1]}{\Vert\phi_1\Vert_{L^2(\mathbb{R}^{d-1})}^2}\geq\frac{\mathfrak{d}_{\aa_0+(\aa_1)_+^*,\lambda_1}[\phi_1^*]}{\Vert\phi_1^*\Vert^2_{L^2(\mathbb{R}^{d-1})}}\geq\mu_{\aa_0+(\aa_1)_+^*}(\lambda_1).
\end{equation*}
Since $\mu_{\aa_0+(\aa_1)_+^*}$ is nonincreasing by Lemma \ref{lem_Eigenvalue_function} it admits a zero
\begin{equation*}
\tilde{\lambda}_1\leq\lambda_1<\left\{\begin{array}{ll} -\frac{\alpha_0^2}{4}, & \text{if }\alpha_0\geq 0, \\ 0, & \text{if }\alpha_0\leq 0. \end{array}\right.
\end{equation*}
Hence $\tilde{\lambda}_1\in\sigma_{\rm d}(\sfA_{\alpha_0+(\alpha_1)_+^*})$ and we have $\lambda_1(\alpha_0+(\alpha_1)_+^*)\leq\tilde{\lambda}_1\leq\lambda_1$, which proves the theorem.
\end{proof}

\begin{bem}
We mention that the above results remain valid for Robin Laplacians on the upper half-space $\mathbb{R}^d_+$. More precisely, if 
$\sfB_\alpha$ denotes the self-adjoint operator in $L^2(\dR^d_+)$ 
associated with the densely defined, symmetric, semibounded, 
and closed form 
\begin{align*}
\mathfrak{b}_\alpha[u,v]&:=(\nabla u,\nabla v)_{L^2(\mathbb{R}^d_+,\mathbb{C}^d)}-\int_{\mathbb{R}^{d-1}}\alpha\,\tau_{\rm D}u\,\overline{\tau_{\rm D}v}\,\dd x, \\
\dom\mathfrak{b}_\alpha&:=H^1(\mathbb{R}^d_+),
\end{align*}
and we replace $\lambda_1(\alpha)=\inf\sigma(\sfA_\alpha)$ by the bottom of the spectrum $\lambda_1(\alpha)\coloneqq\inf\sigma(\sfB_\alpha)$, then Theorem~\ref{thm:main2} and Corollary~\ref{qqq} hold.
\end{bem}

\begin{bem}\label{otherbem}
	Theorem~\ref{thm:main2} can be proved differently using Steiner symmetrization; the following elegant argument was communicated to us recently.
	Consider a nonnegative function $u\colon\dR^d\arr\dR$ such that $\dR^{d-1}\ni x' \mapsto u(x',x_d)$ is vanishing at infinity for all $x_d \in\dR$.
	Following the lines of~\cite[Chapter 6]{B} we recall that the $(d-1,d)$-Steiner symmetrization $u^\sharp$ of the function $u$ is defined as	
	\[
		u^\sharp(x',x_d) := (u^*(\cdot,x_d))(x',x_d),
	\]
	where the symmetric decreasing rearrangement in the right hand side is taken for each $x_d\in\dR$ with respect to first $d-1$ variables. Let the nonnegative function $u_1\in H^1(\dR^d)$ be the normalized ground state of the operator $\sfA_{\aa_0+\aa_1}$. It is not difficult to check that $u_1$ is vanishing at infinity slice-wise in the above sense; cf.~\cite[\S 6.8]{B}. According to~\cite[Theorem 6.8]{B} we have
	\begin{equation}\label{eq:Steiner1}
		\|u_1^\sharp\|_{L^2(\dR^d)} = \|u_1\|_{L^2(\dR^d)} = 1.
	\end{equation}
	In view of~\cite[Theorem 6.19]{B}
	we get $u_1^\sharp \in H^1(\dR^d)$ and
	\begin{equation}\label{eq:Steiner2}
		\|\nabla u_1^\sharp\|_{L^2(\dR^d;\dC^d)}\le \|\nabla u_1\|_{L^2(\dR^d;\dC^d)}.
	\end{equation}
	Lemma~\ref{lem:rearr}\,(iv), (v) and (vi)
	yield
	\begin{equation}\label{eq:Steiner3}
		\int_{\dR^{d-1}}(\alpha_0+(\alpha_1)_+^*)|\tau_{\rm D}u_1^\sharp|^2\dd x \ge
		\int_{\dR^{d-1}}(\alpha_0+\alpha_1)|\tau_{\rm D}u_1|^2\dd x.  
	\end{equation}
	Finally, combining~\eqref{eq:Steiner1},~\eqref{eq:Steiner2}, and~\eqref{eq:Steiner3} we obtain by the min-max principle that
	\[
	\begin{aligned}
		\lm_1(\alpha_0+(\alpha_1)_+^*)
		&\le\fra_{\alpha_0+(\alpha_1)_+^*}[u_1^\sharp]
		= \|\nabla u_1^\sharp\|^2_{L^2(\dR^d;\dC^d)} - \int_{\dR^{d-1}}(\alpha_0+(\alpha_1)_+^*)|\tau_{\rm D}u_1^\sharp|^2\dd x\\
		& \le \fra_{\alpha_0+\alpha_1}[u_1] 
		= \lambda_1(\alpha_0+\alpha_1). 
	\end{aligned}
	\]
\end{bem}

\begin{appendix}

\section{}

In this appendix let again $\Sigma$ be a Lipschitz hypersurface as in \eqref{Eq_Sigma} and assume that $\alpha\in L^p(\Sigma)+L^\infty(\Sigma)$ for some $1<p<\infty$ in $d=2$ and for $d-1\leq p<\infty$ in $d\geq 3$ dimensions, as in \eqref{Eq_alpha}. In this setting we consider the multiplication operator
\begin{equation}\label{Eq_Malpha}
M_\alpha:H^{\frac{1}{2}}(\Sigma)\rightarrow L^2(\Sigma)\quad\text{with}\quad M_\alpha\phi\coloneqq|\alpha|^{\frac{1}{2}}\phi,\quad\phi\in H^{\frac{1}{2}}(\Sigma),
\end{equation}
which plays a crucial role in the well definedness of the form $\mathfrak{a}_\alpha$ in Proposition~\ref{prop_Closed_form} and in the derivation of the essential spectrum in Theorem~\ref{thm:Essential_spectrum}. If, in addition, \eqref{Eq_Decay_property_potential} holds, then it turns out that the operator $M_\alpha$ is compact; for 
the convenience of the reader we will provide a complete proof below. The preparatory 
estimate in Lemma~\ref{lem_Malpha_boundedness} is also used to conclude the semiboundedness of the 
form $\mathfrak{a}_\alpha$ in Proposition \ref{prop_Closed_form}. 

\medskip
 
We also want to mention that we consider Sobolev and Lebesgue spaces on the surface $\Sigma$ in the sense that for every $s>0$ and $q\in[1,\infty]$
\begin{equation}\label{Eq_Spaces_on_the_boundary}
\begin{split}
&\phi\in H^s(\Sigma)\text{ if and only if }\phi\circ\Xi\in H^s(\mathbb{R}^{d-1})\text{ and }\Vert\phi\Vert_{H^s(\Sigma)}\coloneqq\Vert\phi\circ\Xi\Vert_{H^s(\mathbb{R}^{d-1})}, \\
&\phi\in L^q(\Sigma)\text{ if and only if }\phi\circ\Xi\in L^q(\mathbb{R}^{d-1})\text{ and }\Vert\phi\Vert_{L^q(\Sigma)}\coloneqq\Vert\phi\circ\Xi\Vert_{L^q(\mathbb{R}^{d-1})},
\end{split}
\end{equation}
where $\Xi(x)\coloneqq(x,\xi(x))$ is a bijective map from $\mathbb{R}^{d-1}$ onto $\Sigma$.

\begin{lem}\label{lem_Malpha_boundedness}
For every $\varepsilon>0$ there exists some $c_\varepsilon\geq 0$, depending on $\alpha$, such that
\begin{equation}\label{Eq_Malpha_boundedness}
\Vert M_\alpha\phi\Vert^2_{L^2(\Sigma)}\leq\varepsilon^2\Vert\phi\Vert^2_{H^{\frac{1}{2}}(\Sigma)}+c_\varepsilon^2\Vert\phi\Vert^2_{L^2(\Sigma)},\quad\phi\in H^{\frac{1}{2}}(\Sigma).
\end{equation}
\end{lem}

\begin{proof}
We decompose $\alpha\in L^p(\Sigma)+L^\infty(\Sigma)$ into
\begin{equation*}
\alpha=\beta+\gamma,\quad\beta\in L^p(\Sigma),\,\gamma\in L^\infty(\Sigma).
\end{equation*}
Fix $\varepsilon>0$.
Then the integrability condition $\beta\in L^p(\Sigma)$ ensures the existence of some $C_\varepsilon\geq 0$ such that $\beta=\beta_1+\beta_2$, where
\begin{equation*}
\beta_1(x)\coloneqq\left\{\begin{array}{ll} 0, & |\beta(x)|\leq C_\varepsilon, \\ \beta(x), & |\beta(x)|>C_\varepsilon, \end{array}\right.\quad\text{and}\quad\beta_2(x)\coloneqq\left\{\begin{array}{ll} \beta(x), & |\beta(x)|\leq C_\varepsilon, \\ 0, & |\beta(x)|>C_\varepsilon, \end{array}\right. 
\end{equation*}
and
\begin{equation}\label{Eq_Potential_decomposition1}
\Vert\beta_1\Vert_{L^p(\Sigma)}\leq\varepsilon^2.
\end{equation}
We now split $\alpha=\beta_1+(\beta_2+\gamma)$ into a bounded part $\beta_2+\gamma$ and an unbounded remainder $\beta_1$ and estimate both parts separately. For $\beta_1$ we use Hölder's inequality and the estimate \eqref{Eq_Potential_decomposition1} to get
\begin{equation}\label{Eq_beta1}
\big\Vert|\beta_1|^{\frac{1}{2}}\phi\big\Vert^2_{L^2(\Sigma)}\leq\Vert\beta_1\Vert_{L^p(\Sigma)}\Vert\phi\Vert^2_{L^{\frac{2p}{p-1}}(\Sigma)}\leq\varepsilon^2\Vert\phi\Vert^2_{L^{\frac{2p}{p-1}}(\Sigma)}\leq\varepsilon^2c_E^2\Vert\phi\Vert^2_{H^{\frac{1}{2}}(\Sigma)},
\end{equation}
where in the last inequality we additionally used the Sobolev embedding on the surface $\Vert\cdot\Vert_{L^{\frac{2p}{p-1}}(\Sigma)}\leq c_E\Vert\cdot\Vert_{H^{\frac{1}{2}}(\Sigma)}$, which follows from the classical Sobolev embedding theorem \cite[Theorem 8.12.6]{Bh2012} on $\mathbb{R}^{d-1}$ and the definition of the Sobolev and Lebesgue norms in \eqref{Eq_Spaces_on_the_boundary}.

\medskip

On the other hand, $\beta_2+\gamma$ can simply be estimated by
\begin{equation}\label{Eq_beta2}
\big\Vert|\beta_2+\gamma|^{\frac{1}{2}}\phi\big\Vert^2_{L^2(\Sigma)}\leq\big(C_\varepsilon+\Vert\gamma\Vert_{L^\infty(\Sigma)}\big)\Vert\phi\Vert^2_{L^2(\Sigma)}.
\end{equation}
Now the estimate \eqref{Eq_Malpha_boundedness} follows from \eqref{Eq_beta1} and \eqref{Eq_beta2}.
\end{proof}

The next lemma treats the transition from weak $H^\frac{1}{2}$-convergence on $\Sigma$ to strong $L^2$-convergence on subsets of finite measure of $\Sigma$; this observation is 
preparatory for the compactness result in Proposition~\ref{prop_Malpha_compactness}.

\begin{lem}\label{lem_Convergence_on_a_set_of_finite_measure}
For every $\phi_0,(\phi_n)_n\in H^{\frac{1}{2}}(\Sigma)$, the convergence
\begin{equation}\label{Eq_Weak_convergence}
\phi_n\rightharpoonup\phi_0\quad\text{weakly in }H^{\frac{1}{2}}(\Sigma),
\end{equation}
implies for any Borel set $A\subseteq\Sigma$ with finite measure, the convergence
\begin{equation}\label{Eq_Convergence_on_sets_of_finite_measure}
\phi_n\rightarrow\phi_0\quad\text{strongly in }L^2(A).
\end{equation}
\end{lem}

\begin{proof}
In \textit{Step 1} we consider the hyperplane case $\Sigma=\mathbb{R}^{d-1}\times\{0\}\cong\mathbb{R}^{d-1}$. For every $t>0$, we define the mollifier
\begin{equation}\label{Eq_Mollifier}
\varphi_t(x)\coloneqq\frac{1}{(4\pi t)^{\frac{d-1}{2}}}e^{-\frac{|x|^2}{4t}},\quad x\in\mathbb{R}^{d-1}.
\end{equation}
Then by the weak convergence \eqref{Eq_Weak_convergence}, we conclude the pointwise convergence of the convolution
\begin{equation}\label{Eq_Convergence_on_a_set_of_finite_measure_6}
\begin{split}
\lim\limits_{n\rightarrow\infty}(\varphi_t*\phi_n)(x)&=\lim\limits_{n\rightarrow\infty}\big<\varphi_t(x-\,\cdot\,),\phi_n\big>_{H^{-\frac{1}{2}}(\mathbb{R}^{d-1})\times H^{\frac{1}{2}}(\mathbb{R}^{d-1})} \\
&=\big<\varphi_t(x-\,\cdot\,),\phi_0\big>_{H^{-\frac{1}{2}}(\mathbb{R}^{d-1})\times H^{\frac{1}{2}}(\mathbb{R}^{d-1})}=(\varphi_t*\phi_0)(x).
\end{split}
\end{equation}
Since the weakly convergent sequence $(\phi_n)_n$ is bounded, i.e. $\Vert\phi_n\Vert_{H^{\frac{1}{2}}(\mathbb{R}^{d-1})}\leq M$ for some $M\geq 0$, 
we also conclude the uniform boundedness of the convolution
\begin{equation}\label{Eq_Convergence_on_a_set_of_finite_measure_2}
|(\varphi_t*\phi_n)(x)|\leq\Vert\varphi_t\Vert_{L^2(\mathbb{R}^{d-1})}\Vert\phi_n\Vert_{L^2(\mathbb{R}^{d-1})}\leq M\Vert\varphi_t\Vert_{L^2(\mathbb{R}^{d-1})},
\end{equation}
for every $x\in\mathbb{R}^{d-1}$, $n\in\mathbb{N}$. Since $A$ is a set of finite measure, \eqref{Eq_Convergence_on_a_set_of_finite_measure_6} \& \eqref{Eq_Convergence_on_a_set_of_finite_measure_2} are sufficient to apply the dominated convergence theorem, which leads to the norm convergence
\begin{equation}\label{Eq_Convergence_on_a_set_of_finite_measure_3}
\lim\limits_{n\rightarrow\infty}\Vert\varphi_t*(\phi_n-\phi_0)\Vert_{L^2(A)}=0.
\end{equation}
For the Fourier transform of the mollifier \eqref{Eq_Mollifier} we have
\begin{align*}
(\mathcal{F}\varphi_t)(k)&=\frac{1}{(2\pi)^{\frac{d-1}{2}}}\int_{\mathbb{R}^{d-1}}e^{-ikx}\varphi_t(x)\dd x=\frac{1}{(8\pi^2t)^{\frac{d-1}{2}}}\int_{\mathbb{R}^{d-1}}e^{-ikx}e^{-\frac{|x|^2}{4t}}\dd x \\
&=\frac{1}{(8\pi^2t)^{\frac{d-1}{2}}}e^{-t|k|^2}\int_{\mathbb{R}^{d-1}}e^{-\frac{(x+2itk)^2}{4t}}\dd x=\frac{1}{(2\pi)^{\frac{d-1}{2}}}e^{-t|k|^2},\quad k\in\mathbb{R}^{d-1},
\end{align*}
and we use the estimate
\begin{equation*}
\big|1-(2\pi)^{\frac{d-1}{2}}(\mathcal{F}\varphi_t)(k)\big|=1-e^{-t|k|^2}\leq c(t|k|^2)^{\frac{1}{4}}\leq ct^{\frac{1}{4}}(1+|k|^2)^{\frac{1}{4}},\quad k\in\mathbb{R}^{d-1},
\end{equation*}
where $c\coloneqq\sup_{y>0} (1-e^{-y}) y^{-\frac{1}{4}}$. Since the Fourier transform of the convolution can be written as the product $\mathcal{F}(\varphi_t*\phi_n)=(2\pi)^{\frac{d-1}{2}}(\mathcal{F}\varphi_t)(\mathcal{F}\phi_n)$, we can estimate the $L^2$-norm
\begin{equation}\label{Eq_Convergence_on_a_set_of_finite_measure_4}
\begin{split}
\Vert\phi_n-\varphi_t*\phi_n\Vert_{L^2(\mathbb{R}^{d-1})}&=\big\Vert\big(1-(2\pi)^{\frac{d-1}{2}}\mathcal{F}\varphi_t\big)\mathcal{F}\phi_n\big\Vert_{L^2(\mathbb{R}^{d-1})}  \\
&\leq ct^{\frac{1}{4}}\big\Vert(1+|\cdot|^2)^{\frac{1}{4}}\mathcal{F}\phi_n\big\Vert_{L^2(\mathbb{R}^{d-1})} \\
&=ct^{\frac{1}{4}}\Vert\phi_n\Vert_{H^{\frac{1}{2}}(\mathbb{R}^{d-1})}. 
\end{split}
\end{equation}
The inequality \eqref{Eq_Convergence_on_a_set_of_finite_measure_4} of course also holds with $\phi_n$ replaced by $\phi_0$, which leads to the estimate
\begin{equation*}
\Vert\phi_n-\phi_0\Vert_{L^2(A)}\leq ct^{\frac{1}{4}} M+\Vert\varphi_t*(\phi_n-\phi_0)\Vert_{L^2(A)}+ct^{\frac{1}{4}}\Vert\phi_0\Vert_{H^{\frac{1}{2}}(\mathbb{R}^{d-1})},
\end{equation*}
for every $n\in\mathbb{N}$ and $t>0$. The first and third term can be made arbitrary small by the choice of $t>0$ and the second term converges by \eqref{Eq_Convergence_on_a_set_of_finite_measure_3}. This proves the statement of the lemma for $\Sigma\cong\mathbb{R}^{d-1}\times\{0\}$.

\medskip

In \textit{Step 2} we consider the general case of a Lipschitz graph $\Sigma$. By the definition of the boundary spaces \eqref{Eq_Spaces_on_the_boundary}, it follows immediately from the weak convergence \eqref{Eq_Weak_convergence}, that also
\begin{equation*}
\phi_n\circ\Xi\rightharpoonup\phi_0\circ\Xi\quad\text{weakly in }H^{\frac{1}{2}}(\mathbb{R}^{d-1}).
\end{equation*}
Since $A$ has finite measure, the preimage $\Xi^{-1}(A)=\{x\in\mathbb{R}^{d-1} \,|\, \Xi(x)\in A\}$ has finite measure as well, and we conclude from the first step 
\begin{equation*}
\phi_n\circ\Xi\rightarrow\phi_0\circ\Xi\quad\text{strongly in }L^2(\Xi^{-1}(A)).
\end{equation*}
By the definition of the boundary spaces \eqref{Eq_Spaces_on_the_boundary} this implies \eqref{Eq_Convergence_on_sets_of_finite_measure}.
\end{proof}

Next we prove the compactness of the multiplication operator $M_\alpha$ under the decay property \eqref{Eq_Decay_property} of the function $\alpha$. Note that, although stated for $\alpha$, this decay property only affects the $L^\infty$-part of $\alpha$. Any function in $L^p(\mathbb{R}^{d-1})$ satisfies \eqref{Eq_Decay_property} automatically. 

\begin{prop}\label{prop_Malpha_compactness}
Assume that the function  $\alpha$ satisfies
\begin{equation}\label{Eq_Decay_property}
\Set{x\in\Sigma | \vert\alpha(x)\vert>\varepsilon}\text{ has finite measure for every }\varepsilon>0.
\end{equation}
Then the multiplication operator $M_\alpha$ in \eqref{Eq_Malpha} is compact.
\end{prop}

\begin{proof}
From Lemma \ref{lem_Malpha_boundedness} we conclude that
$M_\alpha$ in \eqref{Eq_Malpha} is an everywhere defined and bounded operator. In order to prove that $M_\alpha$ is compact, we verify that for any sequence $\phi_n\rightharpoonup\phi_0$ weakly in $H^{\frac{1}{2}}(\Sigma)$, the sequence $M_\alpha\phi_n\rightarrow M_\alpha\phi_0$ converges strongly in $L^2(\Sigma)$. As in the proof of Lemma \ref{lem_Malpha_boundedness}, let $\varepsilon>0$ and decompose the potential into
\begin{equation*}
\alpha=\beta_1+\beta_2+\gamma.
\end{equation*}
Next, we define the set
\begin{equation}\label{Eq_Convergence_property_of_the_potential_2}
A_\varepsilon\coloneqq\Set{x\in\Sigma | \vert\beta_2(x)\vert>\varepsilon^2}\cup\Set{x\in\Sigma | \vert\gamma(x)\vert>\varepsilon^2}.
\end{equation}
The integrability condition $\beta_2\in L^p(\Sigma)$ implies that the set $\{|\beta_2|>\varepsilon^2\}$ has finite measure. Furthermore, since 
$\{|\gamma|>\varepsilon^2\}\subseteq\{|\beta|>\frac{\varepsilon^2}{2}\}\cup\{|\alpha|>\frac{\varepsilon^2}{2}\}$ it follows from the integrability condition 
$\beta\in L^p(\Sigma)$ and the decay property \eqref{Eq_Decay_property} that $\{|\gamma|>\varepsilon^2\}$ also has finite measure. 
Then Lemma~\ref{lem_Convergence_on_a_set_of_finite_measure} shows
\begin{equation*}
\lim\limits_{n\rightarrow\infty}\Vert\phi_n-\phi_0\Vert_{L^2(A_\varepsilon)}=0.
\end{equation*}
This convergence in particular gives an index $N_\varepsilon\in\mathbb{N}$, such that
\begin{equation}\label{Eq_Convergence_property_of_the_potential_4}
\Vert\phi_n-\phi_0\Vert^2_{L^2(A_\varepsilon)}\leq\frac{\varepsilon^2}{C_\varepsilon+\Vert\gamma\Vert_{L^\infty(\Sigma)}},\quad n\geq N_\varepsilon,
\end{equation}
with $C_\varepsilon$ the cut-off from \eqref{Eq_Potential_decomposition1}. Then the equations \eqref{Eq_beta1} \& \eqref{Eq_Convergence_property_of_the_potential_4}, as well as the fact 
that $|\beta_2+\gamma|\leq C_\varepsilon+\Vert\gamma\Vert_{L^\infty(\Sigma)}$ on $\Sigma$ and $|\beta_2+\gamma|\leq 2 \varepsilon^2$ on $\Sigma\setminus A_\varepsilon$, we can estimate
\begin{align*}
\big\Vert|\alpha|^{\frac{1}{2}}(\phi_n-\phi_0)\big\Vert^2_{L^2(\Sigma)}&\leq\big\Vert|\beta_1|^{\frac{1}{2}}(\phi_n-\phi_0)\big\Vert^2_{L^2(\Sigma)}+\big\Vert|\beta_2+\gamma|^{\frac{1}{2}}(\phi_n-\phi_0)\big\Vert^2_{L^2(A_\varepsilon)} \\
&\quad+\big\Vert|\beta_2+\gamma|^{\frac{1}{2}}(\phi_n-\phi_0)\big\Vert^2_{L^2(\Sigma\setminus A_\varepsilon)} \\
&\leq\varepsilon^2c_E^2\Vert\phi_n-\phi_0\Vert^2_{H^{\frac{1}{2}}(\Sigma)}+\varepsilon^2+2\varepsilon^2\Vert\phi_n-\phi_0\Vert^2_{L^2(\Sigma\setminus A_\varepsilon)} \\
&\leq\varepsilon^2\big((c_E^2+2)\Vert\phi_n-\phi_0\Vert^2_{H^{\frac{1}{2}}(\Sigma)}+1\big),\quad n\geq N_\varepsilon.
\end{align*}
Since $\Vert\phi_n-\phi_0\Vert_{H^{\frac{1}{2}}(\Sigma)}$ on the right hand side is bounded as a consequence of the weak $H^{\frac{1}{2}}$-convergence, this inequality implies the norm convergence
\begin{equation*}
\lim\limits_{n\rightarrow\infty}\big\Vert|\alpha|^{\frac{1}{2}}(\phi_n-\phi_0)\big\Vert^2_{L^2(\Sigma)}=0,
\end{equation*}
and hence the compactness of the operator $M_\alpha$.
\end{proof}

\end{appendix}

\end{document}